\documentclass[12pt,reqno]{amsart}
\usepackage{amsmath,amsthm,amsfonts,amssymb}
\usepackage{graphicx}
\usepackage[lite]{amsrefs}
\usepackage[all]{xy}

\newtheorem{theorem}{Theorem}[section]
\newtheorem{lemma}[theorem]{Lemma}

\newtheorem{proposition}[theorem]{Proposition}

\newcommand{\vkq}[1]{[ #1 ]}
\newcommand{\vkqq}[1]{\left[ #1 \right]}
\newcommand{\vkqqq}[1]{\big[ #1 \big]}
\newcommand{\vkqqqq}[1]{\Big[ #1 \Big]}
\newcommand{\vkw}[1]{\langle #1 \rangle}

\newcommand{\vkwww}[1]{\big\langle #1 \big\rangle}
\newcommand{\vkwwww}[1]{\Big\langle #1 \Big\rangle}
\newcommand{\vkn}{\mathrm{left}}
\newcommand{\vkm}{\mathrm{right}}
\newcommand{\vkord}{\mathrm{ord}}
\newcommand{\vkcomp}{\mathrm{comp}}
\newcommand{\vkparti}{\mathrm{part}}

\numberwithin{equation}{section}

\begin{document}

\title[Bellman function technique for multilinear estimates]
{Bellman function technique for multilinear estimates and an application to generalized paraproducts}
\author{Vjekoslav Kova\v{c}}
\address{Department of Mathematics, UCLA, Los Angeles, CA 90095-1555}
\email{vjekovac@math.ucla.edu}


\begin{abstract}
We prove $\mathrm{L}^p$ estimates for a class of two-dimensional multilinear forms
that naturally generalize (dyadic variants of) both classical paraproducts
and the twisted paraproduct introduced in \cite{DT} and studied in \cite{Be} and \cite{K1}.
The method we use builds on the approach from \cite{K1} and we present it as a rather general
technique for proving estimates on dyadic multilinear operators.
In the particular application to ``generalized paraproducts'' this method is
combined with combinatorics of integer partitions.
\end{abstract}

\maketitle

\section{Introduction and formulation of the result}
\label{vksec1paraproductforms}

\emph{Paraproducts} first appeared in the work of Pommerenke \cite{P}, who showed that
\,$h'=f' g$\, implies \,$\|h\|_{\mathrm{H}^2} \leq C \,\|f\|_{\mathrm{BMOA}} \|g\|_{\mathrm{H}^2}$\,
for analytic functions on the unit disk.
They were named and used extensively by Bony \cite{Bn} in the context of his theory of para\-differential operators.
Since then, many variants have been studied and they have proven to be a useful concept in various mathematical disciplines.
We choose the formulation appearing in \cite{T2}.
A multilinear form is called a \emph{(classical) model paraproduct} if it is given by
$$ \Lambda_{1\mathrm{DP}}(f_1,\ldots,f_n)
= \sum_{I\in\mathcal{D}} \,|I|^{1-\frac{n}{2}}\,
\prod_{i=1}^{n} \big\langle f_i, \varphi^{(i)}_I \big\rangle_{\mathrm{L}^2(\mathbb{R})} \\
= \sum_{I\in\mathcal{D}} \,|I|^{1-\frac{n}{2}}
\int_{\mathbb{R}^n} \prod_{i=1}^{n} f_i(x_i)\varphi^{(i)}_I\!(x_i)dx_i , $$
where $\mathcal{D}$ is the collection of dyadic intervals and each $\varphi^{(i)}_I$ is
a smooth bump function adapted\footnote{Consult \cite{T2} for rigorous definitions of the notions
\emph{bump function} and \emph{adapted}. This paper deals with dyadic models only.}
to the interval $I$.
We also assume that for each $I$ at least two of the functions
$\varphi^{(1)}_I,\ldots,\varphi^{(n)}_I$ have mean zero.
Classical Calder\'{o}n-Zygmund theory establishes the estimate\footnote{For two nonnegative quantities
$A$ and $B$, we write $A\lesssim_P B$ if $A\leq C_P B$ holds for some constant $C_P\geq 0$ depending on
a set of parameters $P$.}
$$ |\Lambda_{1\mathrm{DP}}(f_1,\ldots,f_n)| \,\lesssim_{n,(p_i)}\, \prod_{i=1}^{n}\|f_i\|_{\mathrm{L}^{p_{i}}(\mathbb{R})} $$
for the exponents satisfying $\sum_{i=1}^{n}\frac{1}{p_i}=1$ and $1<p_i<\infty$;\, see \cite{T2}.
One can consult \cite{St} for applications and \cite{Be2},\,\cite{MTT4} for more recent
extensions and references.

Significant conceptual complications that do not seem to have been extensively studied
arise as soon as one proceeds to higher dimensions.
For expositional and notational simplicity we will only work with functions on $\mathbb{R}^2$.
Furthermore, our method significantly uses the dyadic structure, so it is appropriate
to replace each smooth bump function $\varphi^{(i)}_I$
either by the Haar scaling function $\varphi^{\mathrm{D}}_I$ or by the Haar wavelet $\psi^{\mathrm{D}}_I$.
The latter two are defined simply as
$$ \varphi^{\mathrm{D}}_I:=|I|^{-1/2}\mathbf{1}_{I}\quad\textrm{and}\quad
\psi^{\mathrm{D}}_I:=|I|^{-1/2}(\mathbf{1}_{I_\vkn}-\mathbf{1}_{I_\vkm}) . $$
Here, $I_\vkn$ and $I_\vkm$ respectively denote left and right halves of a dyadic interval $I$
and we use the notation $\mathbf{1}_A$ for the characteristic function of a set $A\subseteq\mathbb{R}$.
Note that the function $\psi^{\mathrm{D}}_I$ has mean zero.

An obvious two-dimensional analogue of the \emph{dyadic model paraproduct} is
\begin{align}
& \Lambda_{2\mathrm{DP}}(F_1,\ldots,F_n)
= \sum_{I\times J\in\mathcal{C}} \!|I|^{2-n}
\Big(\prod_{i\in S} \left\langle F_i, \psi^{\mathrm{D}}_I\!\otimes\!\varphi^{\mathrm{D}}_J
\right\rangle_{\mathrm{L}^2(\mathbb{R}^2)} \!\Big)
\Big(\prod_{i\in S^c} \left\langle F_i, \varphi^{\mathrm{D}}_I\!\otimes\!\varphi^{\mathrm{D}}_J
\right\rangle_{\mathrm{L}^2(\mathbb{R}^2)} \!\Big) \nonumber \\[-1.5mm]
& = \sum_{I\times J\in\mathcal{C}} |I|^{2-n} \int_{\mathbb{R}^{2n}}
\Big(\prod_{i\in S} \psi^{\mathrm{D}}_I(x_i)\Big)
\Big(\prod_{i\in S^c} \varphi^{\mathrm{D}}_I(x_i)\Big)
\Big(\prod_{i=1}^{n} F_i(x_i,y_i)\varphi^{\mathrm{D}}_J(y_i)dx_i dy_i\Big) ,
\label{vkeqclassicalpprod}
\end{align}
where $\mathcal{C}$ denotes the collection of all dyadic squares in $\mathbb{R}^2$,
$$ \mathcal{C} := \Big\{ \big[2^{k} l_1,2^{k} (l_1\!+\!1)\big)\times\big[2^{k} l_2,2^{k} (l_2\!+\!1)\big)
\,:\, k,l_1,l_2\in\mathbb{Z} \Big\} , $$
and \,$S\subseteq\{1,2,\ldots,n\}$,\, $|S|\geq 2$,\,
$S^c := \{1,2,\ldots,n\}\setminus S$.
This object still falls into the realm of Calder\'{o}n-Zygmund theory.

More interesting analogue was proposed by Demeter and Thiele in \cite{DT},
while investigating the two-dimensional bilinear Hilbert transform.
They raised a question of proving $\mathrm{L}^p$ estimates for
$$ \Lambda_{\mathrm{TP}}(F,G,H) \!= \!\int_{\mathbb{R}^2} \sum_{k\in\mathbb{Z}} 2^{2k}
\Big(\!\int_{\mathbb{R}}\! F(x\!-\!s,y)\varphi(2^k s)ds\!\Big)
\Big(\!\int_{\mathbb{R}}\! G(x,y\!-\!t)\psi(2^k t)dt\!\Big) H(x,y) dx dy, $$
where $\varphi,\psi$ are smooth bump functions and
\,$\mathrm{supp}(\hat{\psi}) \subseteq \{\xi\in\mathbb{R} \,:\, {\textstyle\frac{1}{2}}\!\leq\! |\xi|\leq 2\}$.\,
After a conditional result by Bernicot \cite{Be}, first $\mathrm{L}^p$ bounds for this multilinear form
were established in \cite{K1}.
The form is called the \emph{twisted paraproduct} and its dyadic variant can be written as
\begin{align}
& \Lambda_{\mathrm{DTP}}(F,G,H) = \sum_{I\times J\in\mathcal{C}}
\int_{\mathbb{R}^4}\! F(x_2,y_1) G(x_1,y_2) H(x_1,y_1) \nonumber \\[-2mm]
& \qquad\qquad\qquad\qquad\qquad\quad\ \
\psi^{\mathrm{D}}_{I}(x_1)\psi^{\mathrm{D}}_{I}(x_2) \varphi^{\mathrm{D}}_{J}(y_1)\varphi^{\mathrm{D}}_{J}(y_2)
\,dx_1 dx_2 dy_1 dy_2 .
\label{vkeqtwistedpprod}
\end{align}
Indeed, the approach in \cite{K1} was to control the dyadic version first and then relate the continuous version
to the dyadic one.

In this paper we define a class of two-dimensional multilinear forms
that naturally generalize both classical paraproducts (\ref{vkeqclassicalpprod})
and the twisted paraproduct (\ref{vkeqtwistedpprod}).
Then we prove $\mathrm{L}^p$ estimates in a certain range of exponents depending on the structure of
the particular multilinear form.
It turns out convenient to associate the new paraproduct-type forms to finite bipartite undirected graphs.
This already indicates a certain amount of combinatorial reasoning appearing in the proof.

\smallskip
Let $m,n$ be positive integers and choose
$$ E\subseteq \{1,\ldots,m\}\times\{1,\ldots,n\}, \quad S\subseteq \{1,\ldots,m\}, \quad T\subseteq \{1,\ldots,n\} . $$
It will be convenient to represent $E$ as the set of edges of a simple bipartite undirected graph
with vertices $\{x_1,\ldots,x_m\}$ and $\{y_1,\ldots,y_n\}$,
where $x_i$ and $y_j$ are connected by an edge if and only if $(i,j)\in E$.
Also, we regard elements of $S$ and $T$ as ``selected'' vertices from these two vertex-sets respectively.
We additionally require \,$|S|\geq 2$\, or \,$|T|\geq 2$.
\vspace*{-4.5mm}
\begin{figure}[th]
\begin{center}\includegraphics[width=0.315\textwidth]{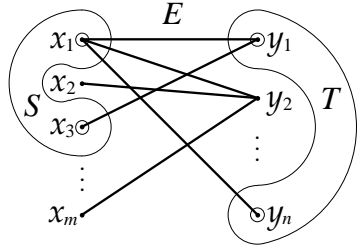}\end{center}
\vspace*{-3mm}
\caption{Graph interpretation of a triple $(E,S,T)$. Selected vertices are circled.}
\label{vkfiguregraph}
\end{figure}

To each such triple $(E,S,T)$ we associate a multilinear form $\Lambda=\Lambda_{E,S,T}$ acting on $|E|$ functions by
\begin{align*}
& \Lambda\big((F_{i,j})_{(i,j)\in E}\big) \ := \sum_{I\times J\in\mathcal{C}}
|I|^{2-\frac{m+n}{2}} \int_{\mathbb{R}^{m+n}} \Big(\prod_{(i,j)\in E}\!\! F_{i,j}(x_i,y_j)\Big) \\
& \qquad\quad\Big(\!\prod_{i\in S}\!\psi^{\mathrm{D}}_I(x_i)\!\Big)\! \Big(\!\prod_{i\in S^c}\!\varphi^{\mathrm{D}}_I(x_i)\!\Big)\!
\Big(\!\prod_{j\in T}\!\psi^{\mathrm{D}}_J(y_j)\!\Big)\! \Big(\!\prod_{j\in T^c}\!\varphi^{\mathrm{D}}_J(y_j)\!\Big)
\,dx_1\ldots dx_m \,dy_1\ldots dy_n .
\end{align*}
To make sure that the above summand is well-defined for each $I\times J\in\mathcal{C}$
and that all of the succeeding arguments are valid,
we suppose that $F_{i,j}$ are measurable, bounded, and compactly
supported functions.
%\footnote{Absolute convergence
%of the series $\sum_{I\times J\in\mathcal{C}}$ will be a part of Theorem \ref{vktheoremmainbound}.}
Since this is only a qualitative assumption, any quantitative bounds can be extended by density arguments.
The normalization $|I|^{2-(m+n)/2}$ is chosen so that $\Lambda$ is invariant under simultaneous dyadic dilations:
$$ \Lambda\big((\mathrm{D}_{2^l}F_{i,j})_{(i,j)\in E}\big) \,=\, 2^{2l}\,\Lambda\big((F_{i,j})_{(i,j)\in E}\big) , $$
where $l\in\mathbb{Z}$ and $(\mathrm{D}_{2^l}F)(x,y):=F(2^{-l}x,2^{-l}y)$.

Informally, we say that the functions appear in a certain ``twisted'', ``entwined'', or ``entangled'' way
in the definition of $\Lambda$.
One can describe the structure of $\Lambda$ in words:
\begin{itemize}
\item[(1)]
Every edge $(x_i,y_j)$ contributes with a function $F_{i,j}$.
\item[(2)]
Each vertex, $x_i$ or $y_j$, carries a ``dyadic bump function'' (either $\varphi^{\mathrm{D}}$ or $\psi^{\mathrm{D}}$).
\item[(3)]
Selected vertices carry ``dyadic bump functions'' of mean zero (i.e.\@ $\psi^{\mathrm{D}}$).
\item[(4)]
At least one bipartition class, $\{x_1,\ldots,x_m\}$ or $\{y_1,\ldots,y_n\}$, contains at least two
selected vertices.
\end{itemize}
The last condition is an analogue of the standard cancellation condition for classical paraproducts.

\smallskip
Now we state a boundedness result for these forms.
The bipartite graph determined by $E$ splits into connected components, i.e.\@ maximal connected subgraphs.
Let $d_{i,j}$ denote larger size of the two bipartition classes of the connected component
containing an edge $(x_i,y_j)$.

\begin{theorem}
\label{vktheoremmainbound}
Let $(E,S,T)$ and $(d_{i,j})_{(i,j)\in E}$ be as above.
The series $\sum_{I\times J\in\mathcal{C}}$ defining $\Lambda$ converges absolutely and
the form $\Lambda$ satisfies the estimate
\begin{equation}
\label{vkeqlpestimate}
\big|\Lambda\big((F_{i,j})_{(i,j)\in E}\big)\big|
\ \lesssim_{m,n,(p_{i,j})} \prod_{(i,j)\in E}\!\! \|F_{i,j}\|_{\mathrm{L}^{p_{i,j}}(\mathbb{R}^2)}
\end{equation}
whenever the exponents $(p_{i,j})_{(i,j)\in E}$ are such that
\,$\sum_{(i,j)\in E}\frac{1}{p_{i,j}}=1$\, and \,$d_{i,j}\!<\!p_{i,j}\!<\!\infty$\, for each $(i,j)\in E$.
\end{theorem}

\smallskip
Let us comment on a couple of already familiar particular instances.
\begin{list}{\labelitemi}{\setlength{\leftmargin}{0em}
\setlength{\itemsep}{5pt}\setlength{\itemindent}{0em}}
\item[]\emph{Classical paraproducts.}\ \,
$m=n$,\, $E=\big\{(i,i) : i\in\{1,\ldots,n\}\big\}$,\, $|S|\geq 2$,\, $T=\emptyset$,\, $d_{i,i}=1$.
\vspace*{-4.5mm}
\begin{figure}[ht]
\begin{center}\includegraphics[width=0.233\textwidth]{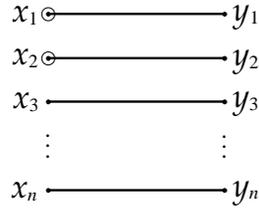}\end{center}
\vspace*{-3.5mm}
\caption{Bipartite graph corresponding to a classical paraproduct.}
\end{figure}

\noindent
This special case leads to ordinary two-dimensional dyadic paraproducts (\ref{vkeqclassicalpprod})
and Theorem \ref{vktheoremmainbound} yields the inequality
$$ |\Lambda_{2\mathrm{DP}}(F_1,\ldots,F_n)| \,\lesssim_{n,(p_i)}\, \prod_{i=1}^{n}\|F_{i}\|_{\mathrm{L}^{p_{i}}(\mathbb{R}^2)} $$
for \,$\sum_{i=1}^{n}\frac{1}{p_i}=1$, \,$1<p_i<\infty$.

\item[]\emph{Twisted paraproduct.}\ \,
$m=n=2$,\,  $E=\big\{(1,1),(1,2),(2,1)\big\}$,\, $S=\{1,2\}$,\,  $T=\emptyset$,\, $d_{1,1}=d_{1,2}=d_{2,1}=2$.
\vspace*{-4.5mm}
\begin{figure}[ht]
\begin{center}\includegraphics[width=0.236\textwidth]{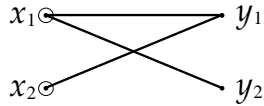}\end{center}
\vspace*{-3.5mm}
\caption{Bipartite graph corresponding to the twisted paraproduct.}
\end{figure}

\noindent
This case is exactly the dyadic variant of the twisted paraproduct (\ref{vkeqtwistedpprod})
and Theorem \ref{vktheoremmainbound} claims the estimate
$$ |\Lambda_{\mathrm{DTP}}(F,G,H)| \,\lesssim_{p,q,r}
\|F\|_{\mathrm{L}^{p}(\mathbb{R}^2)}\|G\|_{\mathrm{L}^{q}(\mathbb{R}^2)}\|H\|_{\mathrm{L}^{r}(\mathbb{R}^2)} $$
in the range \,$\frac{1}{p}+\frac{1}{q}+\frac{1}{r}=1$,\, $2\!<\!p,q,r\!<\!\infty$.\,
``Fiber-wise'' Calder\'{o}n-Zygmund decomposi\-tion of Bernicot \cite{Be}
extends the estimate to the range \,$1\!<\!p,q\!<\!\infty$,\, $2\!<\!r\!\leq\!\infty$\,
and even outside the Banach triangle.
\end{list}

Since our method is naturally adjusted to dyadic scales,
we confine ourselves to bounding dyadic model sums only and do not discuss their continuous analogues.
Also note that the decomposition from \cite{Be} does not apply in general,
so there does not seem to be a pre-existing result allowing an extension of the exponent range.
See Section \ref{vksec6remarks} for further discussion of that matter.

Sections \ref{vksec3singletree} and \ref{vksec4maintheorem} are devoted to the proof of
Theorem \ref{vktheoremmainbound}, while Section \ref{vksec2bellmanfunctions} develops
the general technique we apply in the proof.
To aid understanding we explain the key steps on a concrete example in Section \ref{vksec5example}.
The closing Section \ref{vksec6remarks} contains remarks on limitations of the approach and
further problems.

\section{Bellman functions in multilinear setting}
\label{vksec2bellmanfunctions}

In this section we set up the Bellman function scheme appropriate for proving certain estimates
for multilinear operators acting on two-dimensional functions.
We work in $\mathbb{R}^2$, but the material easily generalizes to higher dimensions.

The Bellman function theory in harmonic analysis was invented by Burkholder \cite{Bu}
and developed in \cite{NT},\,\cite{NTV2},\,\cite{NTV1},
and the subsequent papers by the same authors and their collaborators.
The main difference in our setup
is that we do not insist on optimality conditions and our approach can only be used to establish
positive results about boundedness of operators.
This simplifies the theory, since otherwise the corresponding Bellman functions in the sense of optimal control theory
would necessarily have to be ``infinite-dimensional'', i.e.\@ would not depend on finitely many scalar parameters.

\smallskip
Each dyadic square $Q\in\mathcal{C}$ partitions into four congruent dyadic squares that are called
\emph{children} of $Q$, and conversely, $Q$ is said to be their \emph{parent}.
Our method requires working with structured families of dyadic squares called \emph{finite convex trees}.
A \emph{tree} is a collection $\mathcal{T}$ of dyadic squares such that there exists $Q_\mathcal{T}\in\mathcal{T}$,
called the \emph{root} of $\mathcal{T}$, satisfying $Q\subseteq Q_\mathcal{T}$ for every $Q\in\mathcal{T}$.
A tree $\mathcal{T}$ is \emph{convex} if $Q_1\subseteq Q_2\subseteq Q_3$ and $Q_1,Q_3\in\mathcal{T}$ imply $Q_2\in\mathcal{T}$.
Informally, convex trees ``do not skip any scales''.
A \emph{leaf} of $\mathcal{T}$ is a square that is not contained in $\mathcal{T}$, but its parent is.
The family of leaves of $\mathcal{T}$ will be denoted $\mathcal{L}(\mathcal{T})$.
Note that leaves of any finite convex tree partition the root.
\begin{figure}[th]
\begin{center}
\includegraphics[width=0.26\textwidth]{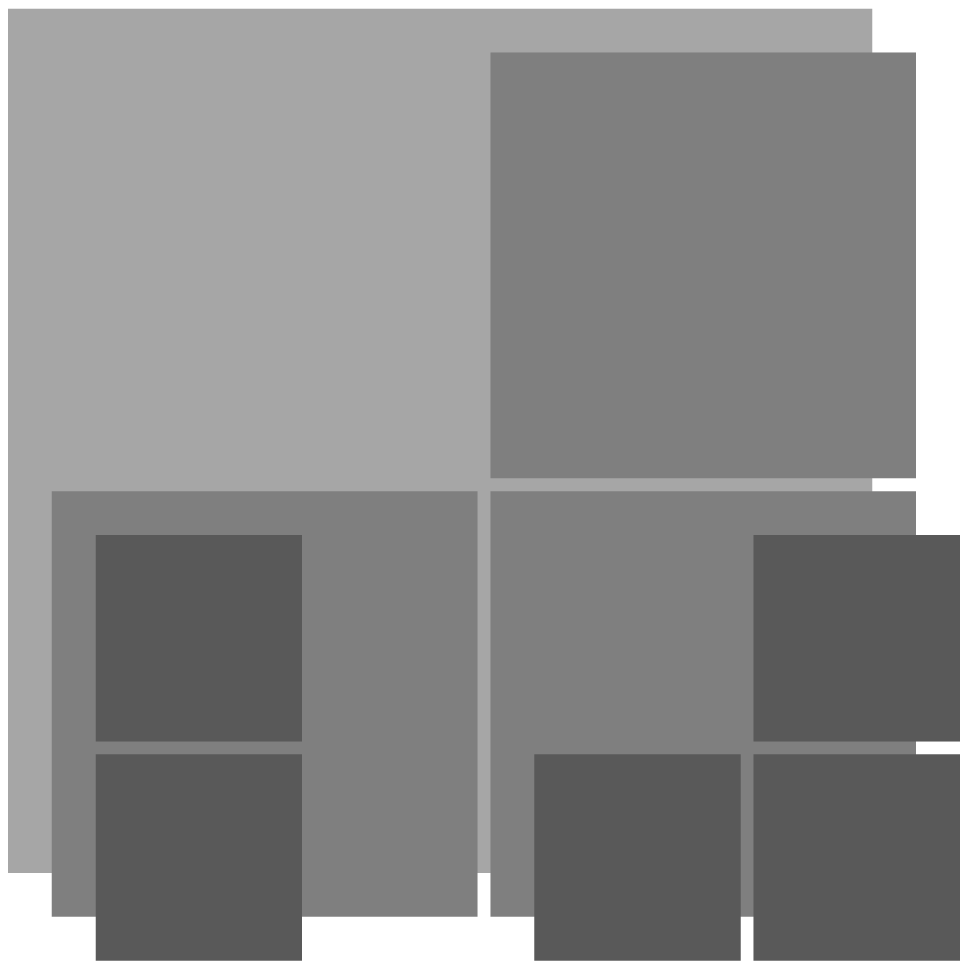}\qquad\qquad
\includegraphics[width=0.26\textwidth]{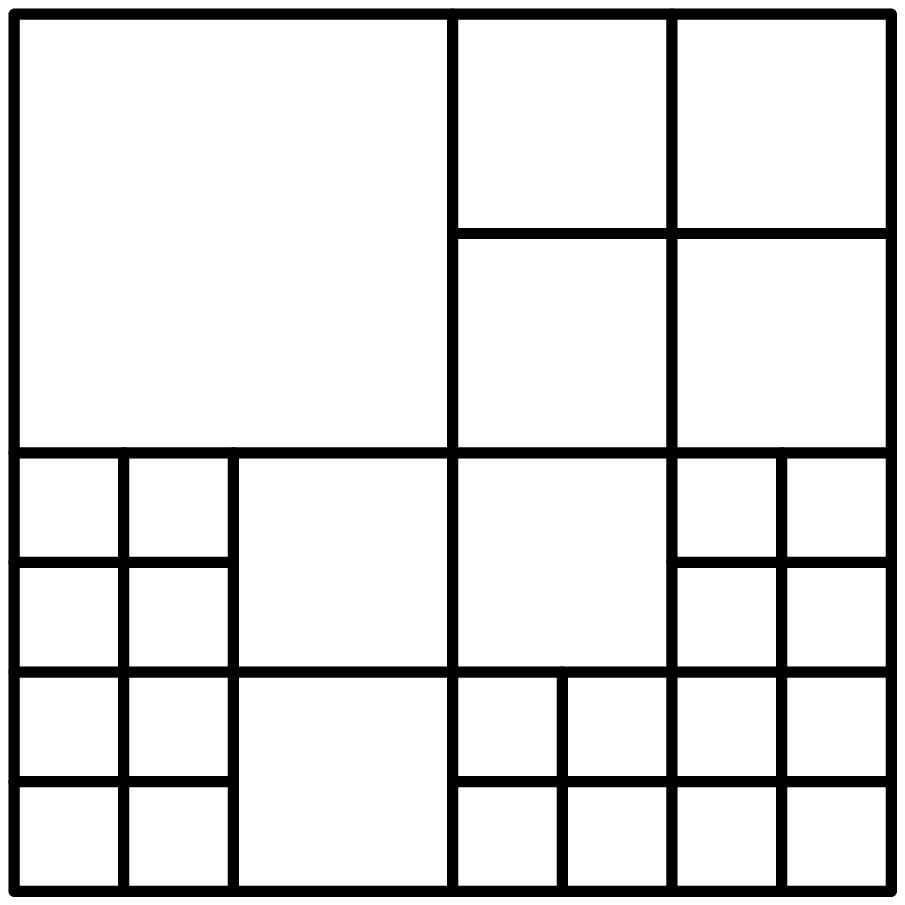}
\vspace*{-2mm}
\caption{A finite convex tree (left) and partition of its root into leaves (right).}
\end{center}
\end{figure}

All of the following functions are assumed to be nonnegative, bounded, and measurable.
For a function $f$ on a dyadic interval $I$ we denote
\begin{align*}
\vkq{f}_I = \vkq{f(x)}_{x\in I} & := \frac{1}{|I|}\int_{I}f(x)dx , \\
\vkw{f}_I = \vkw{f(x)}_{x\in I} & := \frac{1}{|I|}\Big(\int_{I_\vkn}\!\!\!\!f(x)dx - \int_{I_\vkm}\!\!\!\!f(x)dx\Big) .
\end{align*}
We prefer to emphasize the variable in which the average is taken, as for instance
in $\vkq{F(x,y)}_{x\in I}$ and $\vkw{F(x,y)}_{x\in I}$ one deals with functions of more than one variable.
We simply write $\vkq{F(x,y)}_x$ and $\vkw{F(x,y)}_x$ if the interval $I$ is a generic one or is understood.
Notational shortcuts such as
\begin{align*}
\vkw{\Phi(x,x',y)}_{x,x',y} & = \vkwww{\vkwww{\vkwww{\Phi(x,x',y)}_{x\in I}}_{x'\in I}}_{y\in J} , \\
\vkqqq{\Phi(x_1,x_2,\ldots)}_{x_i\in I\textrm{ for }1\leq i\leq n}
& = \vkqqq{\ldots\vkqqq{\vkqqq{\Phi(x_1,x_2,\ldots)}_{x_1\in I}}_{x_2\in I}\ldots}_{x_n\in I}
\end{align*}
are also allowed.
Finally, for a dyadic square $Q=I\times J$ and a function $F$ on it we write
$$ \vkq{F}_Q := \frac{1}{|Q|}\int_{Q}F(x,y)\,dx dy = \vkqqq{F(x,y)}_{x\in I,\, y\in J} . $$

\smallskip
Let us now turn to (multilinear and multi-sublinear) operators we want to study.
A broad class of interesting objects can be reduced to
\begin{equation}
\label{vkeqgeneraldef}
\Lambda_{\mathcal{T}}(F_1,\ldots,F_l) := \sum_{Q\in\mathcal{T}} \,|Q|\,\mathcal{A}_Q (F_1,\ldots,F_l) ,
\end{equation}
where $\mathcal{T}$ is a finite convex tree of dyadic squares and
$\mathcal{A} = \mathcal{A}_Q (F_1,\ldots,F_l)$ is some (usually ``scale-invariant'') quantity
depending on several two-dimensional functions $F_1,\ldots,F_l$ and a square $Q\in\mathcal{T}$.
Several examples are the terms appearing in the right column of Table \ref{vktableexamples}.
Sometimes we want to deal with sums over infinite collections of squares $Q$,
so for this purpose we do not allow any constants to depend on $\mathcal{T}$.

Let $\mathcal{B} = \mathcal{B}_Q (F_1,\ldots,F_l)$ be another quantity
depending on $F_1,\ldots,F_l$ and $Q\in\mathcal{T}$.
We define the \emph{first order difference} of $\mathcal{B}$, denoted
$\Box\mathcal{B}=\Box\mathcal{B}_{Q} (F_1,\ldots,F_l)$, as
\begin{align*}
\Box\mathcal{B}_{I\times J} (F_1,\ldots,F_l)
:= \ & \frac{1}{4}\mathcal{B}_{I_\vkn\times J_\vkn} (F_1,\ldots,F_l)
+ \frac{1}{4}\mathcal{B}_{I_\vkn\times J_\vkm} (F_1,\ldots,F_l) \\
\ + & \frac{1}{4}\mathcal{B}_{I_\vkm\times J_\vkn} (F_1,\ldots,F_l)
+ \frac{1}{4}\mathcal{B}_{I_\vkm\times J_\vkm} (F_1,\ldots,F_l) \\[1mm]
\ - & \mathcal{B}_{I\times J} (F_1,\ldots,F_l) .
\end{align*}
For instance,
\begin{equation}
\label{vkeqsimpleex}
\mathcal{B}_{I\times J} (F) := \vkqqq{\vkq{F(x,y)}_{x\in I}^3}_{y\in J}
\end{equation}
leads to
\begin{align}
\Box\mathcal{B}_{I\times J} (F) =
& \ \frac{1}{4}\vkqqq{\vkq{F(x,y)}_{x\in I_\vkn}^3}_{y\in J_\vkn} + \frac{1}{4}\vkqqq{\vkq{F(x,y)}_{x\in I_\vkn}^3}_{y\in J_\vkm} \nonumber \\
& + \frac{1}{4}\vkqqq{\vkq{F(x,y)}_{x\in I_\vkm}^3}_{y\in J_\vkn} + \frac{1}{4}\vkqqq{\vkq{F(x,y)}_{x\in I_\vkm}^3}_{y\in J_\vkm}
- \vkqqq{\vkq{F(x,y)}_{x\in I}^3}_{y\in J} \nonumber \\
= & \ \frac{1}{2}\vkqqq{\vkq{F(x,y)}_{x\in I_\vkn}^3}_{y\in J} + \frac{1}{2}\vkqqq{\vkq{F(x,y)}_{x\in I_\vkm}^3}_{y\in J}
- \vkqqq{\vkq{F(x,y)}_{x\in I}^3}_{y\in J} \nonumber \\
= & \ 3 \vkqqq{\vkq{F(x,y)}_{x\in I} \vkw{F(x,y)}_{x\in I}^2}_{y\in J} . \label{vkeqsimpleex3}
\end{align}
Above we used obvious identities
\begin{equation*}
\begin{array}{ll}
\vkq{f(x)}_{x\in I} & \!\!\!\!\! = \frac{1}{2}\big(\vkq{f(x)}_{x\in I_\vkn}\!+\vkq{f(x)}_{x\in I_\vkm}\big), \\[1mm]
\vkw{f(x)}_{x\in I} & \!\!\!\!\! = \frac{1}{2}\big(\vkq{f(x)}_{x\in I_\vkn}\!-\vkq{f(x)}_{x\in I_\vkm}\big), \\[1mm]
\vkq{f(x)}_{x\in I_\vkn} & \!\!\!\!\! =\vkq{f(x)}_{x\in I}+\vkw{f(x)}_{x\in I}, \\[1mm]
\vkq{f(x)}_{x\in I_\vkm} & \!\!\!\!\! =\vkq{f(x)}_{x\in I}-\vkw{f(x)}_{x\in I},
\end{array}
\end{equation*}
and $\frac{1}{2}(A\!+\!B)^3+\frac{1}{2}(A\!-\!B)^3 - A^3 = 3A B^2$.
Two more examples are provided in Table \ref{vktableexamples}.
\begin{table}[th]
{\small
\begin{center}\begin{tabular}{|p{0.28\textwidth}|p{0.58\textwidth}|}
\hline \rule{0mm}{4.5mm}
\hspace*{0.125\textwidth}$\mathcal{B}$ & \hspace*{0.268\textwidth}$\Box\mathcal{B}$ \\
\hline \rule{0mm}{5mm}
$\vkqqq{\vkq{F(x,y)}_{x}^3}_{y}$
& $3 \vkqqq{\vkq{F(x,y)}_{x} \vkw{F(x,y)}_{x}^2}_{y}$\\[1mm]
\hline \rule{0mm}{5mm}
$\,\vkq{F(x,y)}_{x,y}^2$
& $\,\vkq{\vkw{F(x,y)}_{x}}_{y}^2 + \vkw{\vkq{F(x,y)}_{x}}_{y}^2 + \vkw{F(x,y)}_{x,y}^2$\\[1mm]
\hline \rule{0mm}{5mm}
$\vkqqq{\vkq{F(x,y)G(x',y)}_{y}^2}_{x,x'}$
& $\vkqqq{\vkw{F(x,y)G(x',y)}_{y}^2}_{x,x'} + \vkwww{\vkq{F(x,y)G(x',y)}_{y}^2}_{x,x'}$ \\[1mm]
& $+\, \vkwww{\vkw{F(x,y)G(x',y)}_{y}^2}_{x,x'}$ \\[1mm]
\hline\end{tabular}\end{center}\rule{0mm}{1mm} }
\caption{Sample table of first order differences.}
\label{vktableexamples}
\end{table}

Let us now suppose that we have found a quantity $\mathcal{B}$ such that
$|\mathcal{A}| \leq \Box\mathcal{B}$, i.e.\@ more precisely
\begin{equation}
\label{vkeqlesstelescope}
|\mathcal{A}_Q (F_1,\ldots,F_l)| \,\leq\, \Box\mathcal{B}_Q (F_1,\ldots,F_l)
\end{equation}
for all squares $Q\in\mathcal{T}$ and any nonnegative bounded functions $F_1,\ldots,F_l$.
By fixing an $l$-tuple of functions, applying (\ref{vkeqlesstelescope}) to an arbitrary $Q\in\mathcal{T}$,
and multiplying by $|Q|$ we get
$$ |Q|\,|\mathcal{A}_Q (F_1,\ldots,F_l)| \,\leq \,
\sum_{\widetilde{Q}\textrm{ is \!a \!child \!of }Q}
\!\!|\widetilde{Q}|\,\mathcal{B}_{\widetilde{Q}} (F_1,\ldots,F_l)
\,-\, |Q|\,\mathcal{B}_{Q} (F_1,\ldots,F_l) . $$
Summing over $Q\in\mathcal{T}$ leads to
\begin{equation}
\label{vkeqmainteleestimate}
|\Lambda_{\mathcal{T}}(F_1,\ldots,F_l)| \,\leq\,
\sum_{Q\in\mathcal{L}(\mathcal{T})} \!|Q|\,\mathcal{B}_Q (F_1,\ldots,F_l)
\,-\, |Q_\mathcal{T}|\,\mathcal{B}_{Q_\mathcal{T}} (F_1,\ldots,F_l)
\end{equation}
for $\Lambda_\mathcal{T}$ given by (\ref{vkeqgeneraldef}).
To verify (\ref{vkeqmainteleestimate}), one only has to notice that each term
$$ |Q|\,\mathcal{B}_{Q} (F_1,\ldots,F_l) \qquad\textrm{for }\, Q\in\mathcal{T}\setminus\{Q_\mathcal{T}\} $$
appears exactly once with a positive sign and exactly once with a negative sign
and thus all terms but those appearing in (\ref{vkeqmainteleestimate}) cancel themselves.
Here is where we crucially use the tree structure --- a general collection of squares would not work.

The quantity $\mathcal{B}$ can be called the \emph{Bellman function} for $\Lambda_{\mathcal{T}}$.
It is certainly not unique and other properties for $\mathcal{B}$ required in the actual problem
will further narrow the choice.
Usefulness of (\ref{vkeqmainteleestimate}) is in the fact that it reduces controlling
a multi-scale quantity $\Lambda_{\mathcal{T}}$ to controlling two single-scale expressions:
one on the level of the ``finest scales'' $\mathcal{L}(\mathcal{T})$
and another one on the level of the ``roughest scale'' $Q_\mathcal{T}$.
Furthermore, if $\mathcal{B}$ is nonnegative, then the last term
$|Q_\mathcal{T}|\,\mathcal{B}_{Q_\mathcal{T}} (F_1,\ldots,F_l)$ can be discarded.

\smallskip
The computation of $\Box\mathcal{B}$ might be quite tedious, even for very simple terms $\mathcal{B}$.
Furthermore, it might not be clear how to find any $\mathcal{B}$
such that $\Box\mathcal{B}$ relates to a given term $\mathcal{A}$.
Fortunately, there is a special class of terms in which both of these tasks are rather straightforward.
A \emph{paraproduct-type term} is a quantity $\mathcal{A}=\mathcal{A}_{I\times J} (F_1,\ldots,F_l)$
that takes the form
\begin{equation}
\label{vkeqstandardgeneral}
\mathcal{A} = \Big(\ldots\Big(\Big(\ldots\Big(\Phi(x_1,\ldots,x_m,y_1,\ldots,y_n)
\Big)_{x_1\in I}\ldots\Big)_{x_m\in I}\Big)_{y_1\in J}\ldots\Big)_{y_n\in J} ,
\end{equation}
where each pair of parentheses $(\cdot)$ is replaced by a pair of brackets, either $\vkq{\cdot}$ or $\vkw{\cdot}$,
with the same subscripted variable.
Here, $\Phi(x_1,\ldots,x_m,y_1,\ldots,y_n)$ depends on the functions $F_1,\ldots,F_l$.
An \emph{averaging paraproduct-type term} is a paraproduct-type term
$\mathcal{B}=\mathcal{B}_{I\times J} (F_1,\ldots,F_l)$
containing only brackets of type $\vkq{\cdot}$, i.e.
\begin{equation}
\label{vkeqstandardterm}
\mathcal{B} = \vkqqqq{\Phi(x_1,\ldots,x_m,y_1,\ldots,y_n)}_{x_1,\ldots,x_m\in I,\ y_1,\ldots,y_n\in J} .
\end{equation}
Linear combinations of paraproduct-type terms are called \emph{paraproduct-type expressions}.
We also regard $\mathcal{B}\equiv 0$ to be a (trivial) averaging paraproduct-type term.
Nontrivial examples of paraproduct-type terms have already existed in the literature.
For instance, the \emph{Gowers box norm} used by Shkredov in \cite{S}
can be written as
$$ \|F\|_{{\setlength{\fboxsep}{0pt}\fbox{\rule{0mm}{1.5mm}\rule{2.8mm}{0mm}\rule{0mm}{1.5mm}}}(I\times J)}
= \vkqqq{F(x_1,y_1) F(x_1,y_2) F(x_2,y_1) F(x_2,y_2)}_{x_1,x_2\in I,\ y_1,y_2\in J}^{1/4} . $$

\begin{theorem}
\label{vktheoremterms}
The first order difference of the averaging paraproduct-type term \emph{(\ref{vkeqstandardterm})}
is a paraproduct-type expression given by the formula
$$ \Box\mathcal{B} = \!\sum_{\substack{S\subseteq\{1,\ldots,m\},\ T\subseteq\{1,\ldots,n\}\\
|S|,|T|\,\mathrm{even},\ (S,T)\neq(\emptyset,\emptyset)}}\!
\vkqqqq{\vkwwww{\Phi(x_1,\ldots,x_m,y_1,\ldots,y_n)
}_{\substack{x_i\in I\ \mathrm{ for }\ i\in S\\ y_j\in J\ \mathrm{ for }\ j\in T}\,}
}_{\substack{x_i\in I\ \mathrm{ for }\ i\in S^c\\ y_j\in J\ \mathrm{ for }\ j\in T^c}} . $$
\end{theorem}

Before the proof, let us remark that we can extend the definition of
paraproduct-type terms to formal finite products consisting of
finitely many two-dimensional functions $F,G,H,F_1,F_2,\ldots$ in finitely many variables
$x,x',y,y',x_1,x_2\ldots$,
with finitely many inserted brackets of two types, $\vkq{\cdot}$ and $\vkw{\cdot}$, each with a variable in its subscript,
as long as they can be transformed into the standard form (\ref{vkeqstandardgeneral}) by renaming duplicate variables.
For example, (\ref{vkeqsimpleex}) is just a shorter form of
$$ \vkqqqq{\vkq{F(x,y)}_{x\in I} \,\vkq{F(x,y)}_{x\in I} \,\vkq{F(x,y)}_{x\in I} }_{y\in J} $$
and can be further rewritten as
\begin{equation}
\label{vkeqsimpleex2}
\vkqqq{ \underbrace{F(x_1,y) F(x_2,y) F(x_3,y)}_{\Phi(x_1,x_2,x_3,y)} }_{x_1,x_2,x_3\in I,\ y\in J} .
\end{equation}
All terms in Table \ref{vktableexamples} are paraproduct-type terms,
but only the terms in the left column are averaging.
The formula from Theorem \ref{vktheoremterms} claims that
$\Box\mathcal{B}$ is equal to the sum of all non-averaging paraproduct-type terms
obtained by replacing some pairs of brackets of type $\vkq{\cdot}$ with pairs of brackets of type $\vkw{\cdot}$
in any possible way such that:
\begin{itemize}
\item[(1)]
The number of replacements corresponding to variables in $I$ is even.
\item[(2)]
The number of replacements corresponding to variables in $J$ is even.
\item[(3)]
At least two replacements are made, i.e.\@ the derived terms are not averaging.
\end{itemize}
In particular, if $\mathcal{B}$ contains $m$ brackets corresponding to variables in $I$
and $n$ brackets corresponding to variables in $J$,
then $\Box\mathcal{B}$ will consist of $2^{m+n-2}-1$ (possibly repeating) terms.
For instance, in (\ref{vkeqsimpleex2}) there are $3$ brackets corresponding to variables in $I$
and $1$ bracket corresponding to a variable in $J$,
giving us only $2^{3+1-2}-1=3$ possibilities.
All $3$ of these terms are equal and we arrive at (\ref{vkeqsimpleex3}).

\begin{proof}[Proof of Theorem \ref{vktheoremterms}]
The claim can be rewritten (using the definition of $\Box\mathcal{B}$) as
\begin{align}
& \hspace{5.25mm}\frac{1}{4}\vkqqq{\Phi(x_1,\ldots,x_m,y_1,\ldots,y_n)}_{x_1,\ldots,x_m\in I_\vkn,\ y_1,\ldots,y_n\in J_\vkn} \nonumber\\
& +\frac{1}{4}\vkqqq{\Phi(x_1,\ldots,x_m,y_1,\ldots,y_n)}_{x_1,\ldots,x_m\in I_\vkn,\ y_1,\ldots,y_n\in J_\vkm} \nonumber\\
& +\frac{1}{4}\vkqqq{\Phi(x_1,\ldots,x_m,y_1,\ldots,y_n)}_{x_1,\ldots,x_m\in I_\vkm,\ y_1,\ldots,y_n\in J_\vkn} \label{vkeqmainthm}\\
& +\frac{1}{4}\vkqqq{\Phi(x_1,\ldots,x_m,y_1,\ldots,y_n)}_{x_1,\ldots,x_m\in I_\vkm,\ y_1,\ldots,y_n\in J_\vkm} \nonumber\\
& = \!\sum_{\substack{S\subseteq\{1,\ldots,m\},\ T\subseteq\{1,\ldots,n\}\\ |S|,|T|\,\mathrm{even}}}\!
\vkqqqq{\vkwwww{\Phi(x_1,\ldots,x_m,y_1,\ldots,y_n)
}_{\substack{x_i\in I \textrm{ for } i\in S\\ y_j\in J \textrm{ for } j\in T}\,}
}_{\substack{x_i\in I \textrm{ for } i\in S^c\\ y_j\in J \textrm{ for } j\in T^c}} . \nonumber
\end{align}
For the purpose of the proof, we denote
\,$\vartheta_I := \mathbf{1}_{I_\vkn}-\mathbf{1}_{I_\vkm}$
and immediately observe that
\begin{equation}
\label{vkeqauxave}
\left.{\setlength{\arraycolsep}{2pt}\begin{array}{rl}
\vkw{f(x)}_{x\in I} & = \vkqq{f(x)\vartheta_I (x)}_{x\in I}, \\[1mm]
\vkq{f(x)}_{x\in I_\vkn} & = \vkqqq{f(x)\big(1+\vartheta_I (x)\big)}_{x\in I}, \\[1mm]
\vkq{f(x)}_{x\in I_\vkm} & = \vkqqq{f(x)\big(1-\vartheta_I (x)\big)}_{x\in I}.
\end{array}}\right\}
\end{equation}
We start from an obvious algebraic identity
$$ \Big(\prod_{i=1}^{m} \big(1+\alpha\,\vartheta_I (x_i)\big)\Big)
\Big(\prod_{j=1}^{n} \big(1+\beta\,\vartheta_J (y_j)\big)\Big)
=\!\sum_{\substack{S\subseteq\{1,\ldots,m\}\\ T\subseteq\{1,\ldots,n\}}}\!\!
\alpha^{|S|}\beta^{|T|} \Big(\prod_{i\in S} \vartheta_I (x_i)\Big) \Big(\prod_{j\in T} \vartheta_J (y_j)\Big), $$
proved simply by multiplying out the product on the left hand side.
We choose four particular values for the parameters $\alpha$ and $\beta$,
$$ (\alpha,\beta) \in \big\{(1,1),(1,-1),(-1,1),(-1,-1)\big\}, $$
add up these four equalities, and divide by $4$.
Finally, we multiply by $\Phi(x_1,\ldots,x_m,$ $y_1,\ldots,y_n)$, average  over $x_1,\ldots,x_m\in I$, $y_1,\ldots,y_n\in J$,
and use (\ref{vkeqauxave}) to get the desired Identity (\ref{vkeqmainthm}).
\end{proof}
Theorem \ref{vktheoremterms} generalizes the ``telescoping identity'' from \cite{K1}.
In order to control a given paraproduct-type term $\mathcal{A}$, we first try to dominate it
by ``simpler'' terms, where a certain notion of ``complexity'' is defined depending on the numbers of
different functions/variables involved, or in a more refined way as in Section \ref{vksec3singletree}.
If that is not possible, then we take $\mathcal{B}$ such that $\mathcal{A}$ appears as one of
the summands in $\Box\mathcal{B}$ and try to dominate the remaining terms in $\Box\mathcal{B}$.

Let us state a part of the presented proof as a separate lemma.
\begin{lemma}
\label{vklemmaterms}
The following identity holds:
\begin{align}
& \sum_{\substack{S\subseteq\{1,\ldots,m\}\\ |S|\,\mathrm{even}}}\!
\vkqqqq{\vkwwww{\Psi(x_1,x_2,\ldots,x_m)
}_{x_i\in I\ \mathrm{for}\ i\in S}\,}_{x_i\in I\ \mathrm{for}\ i\in S^c}
\label{vkeqnonnegativesum} \\
& = \frac{1}{2}\vkqqq{\Psi(x_1,x_2,\ldots,x_m)}_{x_1,x_2,\ldots,x_m\in I_\vkn}
+ \frac{1}{2}\vkqqq{\Psi(x_1,x_2,\ldots,x_m)}_{x_1,x_2,\ldots,x_m\in I_\vkm} . \nonumber
\end{align}
In particular, if \,$\Psi(x_1,x_2,\ldots,x_m)\geq 0$,
then the sum \emph{(\ref{vkeqnonnegativesum})} will also be nonnegative.
\end{lemma}
We remark that (\ref{vkeqnonnegativesum}) is constructed by adding up all terms of the form
$$ \Big(\ldots\Big(\Big(\Psi(x_1,x_2,\ldots,x_m)\Big)_{x_1\in I}\Big)_{x_2\in I}\ldots\Big)_{x_m\in I} $$
where an even number of pairs of parentheses $(\cdot)$ is replaced with pairs of brackets $\langle\cdot\rangle$
and the remaining ones are replaced with pairs of brackets $[\cdot]$.
%(Possible repeating terms are added multiple times.)
\begin{proof}[Proof of Lemma \ref{vklemmaterms}]
This time we use a simpler identity,
$$ \prod_{i=1}^{m} \big(1+\alpha\,\vartheta_I (x_i)\big)
=\!\sum_{S\subseteq\{1,\ldots,m\}}\!\!\alpha^{|S|}\,\prod_{i\in S} \vartheta_I (x_i), $$
with $\alpha=\pm 1$.\,
The sum in question is
\begin{align*}
& \sum_{\substack{S\subseteq\{1,\ldots,m\}\\ |S|\,\mathrm{even}}}\!
\vkqqqq{\Psi(x_1,x_2,\ldots,x_m)\prod_{i\in S} \vartheta_I (x_i)}_{x_1,x_2,\ldots,x_m\in I} \\
& = \frac{1}{2}\,\vkqqqq{\Psi(x_1,x_2,\ldots,x_m)\prod_{i=1}^{m}\! \big(1\!+\!\vartheta_I (x_i)\big)}_{x_1,x_2,\ldots,x_m\in I} \\
& \ \ \ + \frac{1}{2}\,\vkqqqq{\Psi(x_1,x_2,\ldots,x_m)\prod_{i=1}^{m}\! \big(1\!-\!\vartheta_I (x_i)\big)}_{x_1,x_2,\ldots,x_m\in I} \\
& = \frac{1}{2}\vkqqq{\Psi(x_1,x_2,\ldots,x_m)}_{x_1,x_2,\ldots,x_m\in I_\vkn}
+ \frac{1}{2}\vkqqq{\Psi(x_1,x_2,\ldots,x_m)}_{x_1,x_2,\ldots,x_m\in I_\vkm} . \tag*{\qedhere}
\end{align*}
\end{proof}

Theorem \ref{vktheoremterms} and Lemma \ref{vklemmaterms} will be applied successively
in Sections \ref{vksec3singletree} and \ref{vksec5example}.

\section{A single tree estimate}
\label{vksec3singletree}

The form $\Lambda=\Lambda_{E,S,T}$ from Section \ref{vksec1paraproductforms}
can be rewritten using the notation from the previous section as
$$ \Lambda\big((F_{i,j})_{(i,j)\in E}\big)\,
= \sum_{Q\in\mathcal{C}} \,|Q| \,\mathcal{A}_{Q}\big((F_{i,j})_{(i,j)\in E}\big) , $$
with
$$ \mathcal{A}_{I\times J}\big((F_{i,j})_{(i,j)\in E}\big) =
\vkqqqq{\vkwwww{\prod_{(i,j)\in E}\!\! F_{i,j}(x_i,y_j)
}_{\substack{x_i\in I \textrm{ for } i\in S\\ y_j\in J \textrm{ for } j\in T}\,}
}_{\substack{x_i\in I \textrm{ for } i\in S^c\\ y_j\in J \textrm{ for } j\in T^c}} . $$
This formulation also justifies the factor $|I|^{2-(m+n)/2}$ in the definition.
We may suppose that all the functions $F_{i,j}$ are nonnegative,
since otherwise they can be split into positive and negative (real and imaginary) parts
and one uses multilinearity of the form.

We need to introduce a local version of $\Lambda$, so
for a finite convex tree of dyadic squares $\mathcal{T}$ we define
$$ \Lambda_{\mathcal{T}}\big((F_{i,j})_{(i,j)\in E}\big)
:= \sum_{Q\in\mathcal{T}} |Q| \,\big|\mathcal{A}_{Q}\big((F_{i,j})_{(i,j)\in E}\big)\big| . $$
Notice an absolute value in the definition, which makes $\Lambda_{\mathcal{T}}$ only multi-sublinear.

The main step towards the proof of Theorem \ref{vktheoremmainbound}
is an estimate for a single tree.

\begin{proposition}[Single tree estimate]
\label{vkpropositiontree}
For any finite convex tree $\mathcal{T}$ with root $Q_\mathcal{T}$ and leaves $\mathcal{L}(\mathcal{T})$ we have
\begin{equation}
\label{vkeqsingletree}
\Lambda_{\mathcal{T}}\big((F_{i,j})_{(i,j)\in E}\big)
\ \lesssim_{m,n} |Q_\mathcal{T}|\prod_{(i,j)\in E} \max_{Q\in\mathcal{T}\cup\mathcal{L}(\mathcal{T})}
\vkqqq{F_{i,j}^{d_{i,j}}}_{Q}^{1/d_{i,j}} .
\end{equation}
\end{proposition}
Note that the implicit constant is independent of the tree $\mathcal{T}$ and the functions $F_{i,j}$.
We can allow it to depend on the graph, because for each pair $(m,n)$ there are only finitely many choices for $(E,S,T)$.
Moreover, $d_{i,j}$ are determined by the graph.

By homogeneity of Estimate (\ref{vkeqsingletree}) and its invariance under dyadic dilations we can normalize
the tree and the functions by $|Q_\mathcal{T}|=1$ and
\begin{equation}
\label{vkeqfnormalize}
\max_{Q\in\mathcal{T}\cup\mathcal{L}(\mathcal{T})}\!\vkqqq{F_{i,j}^{d_{i,j}}}_{Q}^{1/d_{i,j}}=1
\qquad\textrm{for every $(i,j)\in E$}.
\end{equation}
Thus, our task is to prove
\begin{equation}
\label{vkeqsingletreenormalized}
\Lambda_{\mathcal{T}}\big((F_{i,j})_{(i,j)\in E}\big) \,\lesssim_{m,n} 1  .
\end{equation}

\subsection{Proof of Proposition \ref{vkpropositiontree} for complete bipartite graphs}
\label{vksubsectioncomplete}

Most of our effort will be spent in this special case, when
$E=\{1,\ldots,m\}\times\{1,\ldots,n\}$.
In particular, there is only one connected component and
$d_{i,j}=\max\{m,n\}$.
Later we will reduce the general case to this one.

We begin with a simple estimate for functions on a single square.
\begin{lemma}
\label{vklemmahoelder}
For nonnegative functions $G_{i,j}$ on a dyadic square $I\times J$ the following inequality holds,
$$ \bigg[\prod_{\substack{1\leq i\leq m \\ 1\leq j\leq n}}
G_{i,j}(x_i,y_j)\bigg]_{\substack{x_1,\ldots,x_m\in I\\ y_1,\ldots,y_n\in J}}
\leq \ \prod_{\substack{1\leq i\leq m \\ 1\leq j\leq n}}
\vkqqqq{G_{i,j}^{\max\{m,n\}}}_{I\times J}^{1/\max\{m,n\}}. $$
\end{lemma}
\begin{proof}[Proof of Lemma \ref{vklemmahoelder}.]
Because of the obvious symmetry, we can assume $m\geq n$.
With two applications of (generalized) H\"{o}lder's inequality for $n$ and $m$ functions respectively,
we estimate the left hand side as
\begin{align*}
& \vkqqqq{\prod_{i=1}^{m}\vkqqqq{\prod_{j=1}^{n}G_{i,j}(x_i,y_j)}_{x_i}}_{y_1,\ldots,y_n}
\leq \ \vkqqqq{\prod_{i=1}^{m}\prod_{j=1}^{n}\vkqqq{G_{i,j}(x_i,y_j)^n}_{x_i}^{1/n}}_{y_1,\ldots,y_n}\\
& = \prod_{j=1}^{n}\vkqqqq{\prod_{i=1}^{m}\vkqqq{G_{i,j}(x,y)^n}_{x}^{1/n}}_{y}
\leq \ \prod_{j=1}^{n}\prod_{i=1}^{m}\vkqqq{\vkqqq{G_{i,j}(x,y)^n}_{x}^{m/n}}_{y}^{1/m}.
\end{align*}
By Jensen's inequality for the power function with exponent $\frac{m}{n}\geq 1$ we have
$$ \vkqqq{G_{i,j}(x,y)^n}_{x}^{m/n} \leq \vkqqq{G_{i,j}(x,y)^m}_{x} , $$
so
$$ \vkqqq{\vkqqq{G_{i,j}(x,y)^n}_{x}^{m/n}}_{y}^{1/m} \leq \vkqqq{G_{i,j}(x,y)^m}_{x,y}^{1/m} , $$
which completes the proof.
\end{proof}

In the following discussion, all constructions and (implicit) constants are understood to depend on $m$ and $n$,
but not on the tree $\mathcal{T}$ or the functions $F_{i,j}$.
We introduce the notion of a \emph{selective $(m,n)$-partition} as a $(2m+2n)$-tuple of integers
\begin{equation}
\label{vkeqptuple}
\mathbf{p} = (a_1,\ldots,a_m;\, b_1,\ldots,b_n;\, \alpha_1,\ldots,\alpha_m;\, \beta_1,\ldots,\beta_n)
\end{equation}
satisfying:
\begin{itemize}
\item[(1)]
$0\leq\alpha_i\leq a_i$\, for \,$i=1,\ldots,m$ \ and \ $0\leq\beta_j\leq b_j$\, for \,$j=1,\ldots,n$,
\item[(2)]
$a_1+\ldots+a_m=m$\, and \,$b_1+\ldots+b_n=n$,
\item[(3)]
$\alpha_1+\ldots+\alpha_m$\, and \,$\beta_1+\ldots+\beta_n$\, are even,
\item[(4)]
$\alpha_1+\ldots+\alpha_m\neq 0$\, or \,$\beta_1+\ldots+\beta_n\neq 0$.
\end{itemize}
The set of all selective $(m,n)$-partitions will be denoted $\Omega_{m,n}$.
To every $\mathbf{p}\in\Omega_{m,n}$ we associate a paraproduct-type term
$\mathcal{A}^{(\mathbf{p})}=\mathcal{A}_{I\times J}^{(\mathbf{p})}\big((F_{i,j})_{(i,j)\in E}\big)$ by
$$ \mathcal{A}_{I\times J}^{(\mathbf{p})} :=
\bigg[\bigg\langle\prod_{\substack{1\leq i\leq m\\ 1\leq j\leq n}}
\prod_{\substack{1\leq \mu\leq a_i\\ 1\leq \nu\leq b_j}}\! F_{i,j}(x_i^{(\mu)}\!,y_j^{(\nu)})
\bigg\rangle_{\substack{x_i^{(\mu)}\in I \textrm{ for all }(i,\mu)\\ \textrm{such that } 1\leq \mu\leq\alpha_i\\
y_j^{(\nu)}\in J \textrm{ for all }(j,\nu)\\ \textrm{such that } 1\leq \nu\leq\beta_j}\ }
\bigg]_{\substack{x_i^{(\mu)}\in I \textrm{ for all }(i,\mu)\\ \textrm{such that } \alpha_i+1\leq \mu\leq a_i\\
y_j^{(\nu)}\in J \textrm{ for all }(j,\nu)\\ \textrm{such that } \beta_j+1\leq \nu\leq b_j}} . $$
Pairs $(i,j)$ with $a_i=0$ or $b_j=0$ do not exist in the above product
as we interpret (sub)products over empty ranges to be identically $1$.
In words, we average the product of $mn$ terms of the pattern $F_{i,j}(x_i,y_j)$ that contains precisely
$a_i$ copies of $x_i$ and $b_j$ copies of $y_j$.
Averages of type $\vkw{\cdot}_{x_i}$ are taken over $\alpha_i$ of the $x_i$'s (i.e.\@ these $x_i$'s are ``selected''),
while averages of type $\vkq{\cdot}_{x_i}$ are taken over $a_i\!-\!\alpha_i$ remaining ones.
Similarly for $y_j$'s.
For instance, to
$$ \mathbf{p} = (2,0;\ 2,0,1;\ 0,0;\ 1,0,1) \in \Omega_{2,3} $$
we associate
$$ \mathcal{A}^{(\mathbf{p})}
\!=\! \big[\big\langle F_{1,1}(x_1,y_1) F_{1,1}(x_1,y'_1) F_{1,3}(x_1,y_3)
F_{1,1}(x'_1,y_1) F_{1,1}(x'_1,y'_1) F_{1,3}(x'_1,y_3)
\big\rangle_{\!y_1,y_3}\big]_{x_1,x'_1,y'_1}\!. $$

For $\mathbf{p}\in\Omega_{m,n}$ given by (\ref{vkeqptuple}) we define
the \emph{composition type} of $\mathbf{p}$ to be the vector of first $m+n$ components,
$$ \vkcomp(\mathbf{p}) := (a_1,\ldots,a_m;\, b_1,\ldots,b_n) , $$
and the \emph{partition type} of $\mathbf{p}$ (and $\mathcal{A}^{(\mathbf{p})}$) to be an $(m+n)$-tuple
$$ \vkparti(\mathbf{p}) := (a^{\ast}_1,\ldots,a^{\ast}_m;\, b^{\ast}_1,\ldots,b^{\ast}_n) , $$
where $a^{\ast}_1,\ldots,a^{\ast}_m$ is the decreasing rearrangement\footnote{This means:
$a^{\ast}_1\geq\ldots\geq a^{\ast}_m$ and $a^{\ast}_1,\ldots,a^{\ast}_m$
is a permutation of the multiset $a_1,\ldots,a_m$.}
of $a_1,\ldots,a_m$ and $b^{\ast}_1,\ldots,b^{\ast}_n$ is the decreasing rearrangement of $b_1,\ldots,b_n$.
The set of all these partition types will be denoted $\Omega^{\ast}_{m,n}$.
Note that $\Omega^{\ast}_{m,n}$ has $\mathfrak{p}_{\#}(m)\mathfrak{p}_{\#}(n)$ elements,
where $\mathfrak{p}_{\#}(n)$ denotes the number of distinct order-independent positive integer partitions of $n$,
i.e.\@ the number of Young diagrams with $n$ boxes.
Actually, we only use that the cardinalities $|\Omega_{m,n}|$ and $|\Omega^{\ast}_{m,n}|$
are finite numbers depending solely on $m,n$.

Finally, we define a strict total order relation $\prec$ on $\Omega^{\ast}_{m,n}$
simply as the restriction of the inverse of the lexicographical order
on $(m+n)$-tuples of integers.\footnote{Lexicographical order
on partitions of a single positive integer extends the common \emph{dominance order}, which is only a partial order.
Even though the latter one is already strong enough for intended purpose,
we prefer to have linear order to avoid invoking well-founded induction
in the proof. For the same reason we decide to order pairs of partitions totally (for two numbers $m$ and $n$),
although we will only need to compare partitions of a single number.}
Since every finite totally ordered set is isomorphic to an initial segment of positive integers,
we have a natural rank (i.e.\@ order) function,
$\vkord\colon\Omega^{\ast}_{m,n}\to\{1,2,\ldots,\mathfrak{p}_{\#}(m)\mathfrak{p}_{\#}(n)\}$.
Let us simply write \,$\vkord(\mathbf{p})$\, for \,$\vkord\big(\vkparti(\mathbf{p})\big)$.\,
For example, the total order on $\Omega^{\ast}_{2,3}$ and its rank function are
$$ {\setlength\arraycolsep{0pt}\begin{array}{ccccccccccc}
(2,0; 3,0,0) & \prec & (2,0; 2,1,0) & \prec & (2,0; 1,1,1) & \prec &
(1,1; 3,0,0) & \prec & (1,1; 2,1,0) & \prec & (1,1; 1,1,1).\\
\vkord=1 & & \vkord=2 & & \vkord=3 & & \vkord=4 & & \vkord=5 & & \vkord=6
\end{array}} $$

Our goal is to dominate all terms $\mathcal{A}^{(\mathbf{p})}$ by $\Box\mathcal{B}$
for some averaging paraproduct-type expression $\mathcal{B}=\mathcal{B}_{Q}\big((F_{i,j})_{(i,j)\in E}\big)$
that is controlled in the sense
\begin{equation}
\label{vkeqbcontrol}
\max_{Q\in\mathcal{T}\cup\mathcal{L}(\mathcal{T})}\big|\mathcal{B}_{Q}\big((F_{i,j})_{(i,j)\in E}\big)\big|\lesssim_{m,n} 1 .
\end{equation}
This expression $\mathcal{B}$ will be the desired Bellman function.
The goal will be achieved by mathematical induction on $\vkord(\mathbf{p})$
and for this we will need the following crucial reduction lemma.

\begin{lemma}
\label{vklemmareduction}
For any $\mathbf{p}\in\Omega_{m,n}$ there exists an averaging paraproduct-type term
$\mathcal{B}^{(\mathbf{p})}=\mathcal{B}_{Q}^{(\mathbf{p})}\big((F_{i,j})_{(i,j)\in E}\big)$
satisfying \emph{(\ref{vkeqbcontrol})} and such that for any $0<\delta<1$ we have the estimate
$$ |\mathcal{A}^{(\mathbf{p})}| \ \leq\ \Box\mathcal{B}^{(\mathbf{p})}
\ +\ C_{m,n}\,\delta^{-1}\!\!\!\!\!\!\sum_{\substack{\widetilde{\mathbf{p}}\in\Omega_{m,n}\\
\vkord(\widetilde{\mathbf{p}})<\vkord(\mathbf{p})}}
\!\!\!\!\!\!\!|\mathcal{A}^{(\widetilde{\mathbf{p}})}|
\ +\ C_{m,n}\,\delta\!\!\!\!\!\!\sum_{\substack{\widetilde{\mathbf{p}}\in\Omega_{m,n}\\
\vkord(\widetilde{\mathbf{p}})\geq\vkord(\mathbf{p})}}
\!\!\!\!\!\!\!|\mathcal{A}^{(\widetilde{\mathbf{p}})}| $$
with some constant $C_{m,n}>0$.
\end{lemma}
\begin{proof}[Proof of Lemma \ref{vklemmareduction}.]
We distinguish two cases depending on positions of selected vertices,
i.e.\@ on the last $m+n$ coordinates in (\ref{vkeqptuple}).

\begin{list}{\labelitemi}{\setlength{\leftmargin}{0em}
\setlength{\itemsep}{5pt}\setlength{\itemindent}{0em}}
\item[]\emph{Case 1.} \
$\alpha_{i}\neq 0$ for at least two indices $i\in\{1,\ldots,m\}$
\,or\, $\beta_{j}\neq 0$ for at least two indices $j\in\{1,\ldots,n\}$.

By symmetry we may assume that $\alpha_1,\alpha_2\geq 1$ and $a_1\geq a_2$.
In this case we simply take $\mathcal{B}^{(\mathbf{p})}\equiv 0$.\,
Using $|\vkw{f(y)}_{y}|\leq\vkq{|f(y)|}_{y}$ and
$|AB|\leq\frac{1}{2\delta}A^2+\frac{\delta}{2}B^2$ we estimate:
\begin{align*}
|\mathcal{A}^{(\mathbf{p})}| & \leq \bigg[
\,\Big|\vkwwww{\prod_{\substack{1\leq j\leq n\\ 1\leq \nu\leq b_j}}\!\! F_{1,j}(x_1,y_j^{(\nu)})}_{\!x_1}
\vkwwww{\prod_{\substack{1\leq j\leq n\\ 1\leq \nu\leq b_j}}\!\! F_{2,j}(x_2,y_j^{(\nu)})}_{\!x_2}\Big| \\[-2mm]
& \qquad\qquad\qquad\qquad\qquad\qquad\quad\prod_{(i,\mu)\neq (1,1),(2,1)}\!\!
\vkqqqq{\prod_{\substack{1\leq j\leq n\\ 1\leq \nu\leq b_j}}\!\! F_{i,j}(x_i^{(\mu)}\!,y_j^{(\nu)})}_{x_i^{(\mu)}}
\bigg]_{\textrm{all }y_j^{(\nu)}} \\
& \leq \,\frac{1}{2\delta}\, \bigg[
\vkwwww{\prod_{\substack{1\leq j\leq n\\ 1\leq \nu\leq b_j}}\!\! F_{1,j}(x_1,y_j^{(\nu)})}_{\!x_1}^2
\!\prod_{(i,\mu)\neq (1,1),(2,1)}\!\!
\vkqqqq{\prod_{\substack{1\leq j\leq n\\ 1\leq \nu\leq b_j}}\!\! F_{i,j}(x_i^{(\mu)}\!,y_j^{(\nu)})}_{x_i^{(\mu)}}
\bigg]_{\textrm{all }y_j^{(\nu)}} \\
& \ \ \,+ \frac{\delta}{2}\, \bigg[
\vkwwww{\prod_{\substack{1\leq j\leq n\\ 1\leq \nu\leq b_j}}\!\! F_{2,j}(x_2,y_j^{(\nu)})}_{\!x_2}^2
\!\prod_{(i,\mu)\neq (1,1),(2,1)}\!\!
\vkqqqq{\prod_{\substack{1\leq j\leq n\\ 1\leq \nu\leq b_j}}\!\! F_{i,j}(x_i^{(\mu)}\!,y_j^{(\nu)})}_{x_i^{(\mu)}}
\bigg]_{\textrm{all }y_j^{(\nu)}} \\
& = \frac{1}{2}\Big(\delta^{-1}\mathcal{A}^{(\widetilde{\mathbf{p}})}
+ \delta\,\mathcal{A}^{(\overline{\mathbf{p}})}\Big) .
\end{align*}
Here, $\widetilde{\mathbf{p}},\,\overline{\mathbf{p}}\in\Omega_{m,n}$ are defined as follows.
If $\mathbf{p}$ is given by a $(2m+2n)$-tuple (\ref{vkeqptuple}),
then $\widetilde{\mathbf{p}}$ and $\overline{\mathbf{p}}$ will have coordinates:
$$ \begin{array}{l}\begin{array}{ll}
\widetilde{a}_i = \left\{\rule{0cm}{8mm}\right.\!\!\begin{array}{cl}
a_1+1, & \textrm{ for }i=1, \\ a_2-1, & \textrm{ for }i=2, \\ a_i, & \textrm{ for }i\neq 1,2, \end{array}
& \overline{a}_i = \left\{\rule{0cm}{8mm}\right.\!\!\begin{array}{cl}
a_1-1, & \textrm{ for }i=1, \\ a_2+1, & \textrm{ for }i=2, \\ a_i, & \textrm{ for }i\neq 1,2, \end{array} \\
\widetilde{\alpha}_i = \left\{\rule{0cm}{5.5mm}\right.\!\!\begin{array}{cl}
2, & \textrm{ for }i=1, \\ 0, & \textrm{ for }i\neq 1, \end{array}
& \overline{\alpha}_i = \left\{\rule{0cm}{5.5mm}\right.\!\!\begin{array}{cl}
2, & \textrm{ for }i=2, \\ 0, & \textrm{ for }i\neq 2, \end{array} \end{array} \\
\ \,\widetilde{b}_j = \overline{b}_j = b_j,\ \,
\widetilde{\beta}_j = \overline{\beta}_j = 0\ \,\textrm{ for every }j .
\end{array} $$
Observe that $a_1$ appears to the left from $a_2$ in the list $a^{\ast}_1\geq\ldots\geq a^{\ast}_m$.
Therefore simultaneously increasing $a_1$ by $1$ and decreasing $a_2$ by $1$ we produce
a lexicographically larger partition of $m$, so we conclude
$\vkord(\widetilde{\mathbf{p}})<\vkord(\mathbf{p})$.
On the other hand, both \,$\vkord(\overline{\mathbf{p}})<\vkord(\mathbf{p})$\, and
\,$\vkord(\overline{\mathbf{p}})\geq\vkord(\mathbf{p})$\, are possible,
where in the former case we use $\delta<\delta^{-1}$.

\item[]\emph{Case 2.} \
$\alpha_{i}\neq 0$ for at most one index $i\in\{1,\ldots,m\}$
\,and\, $\beta_{j}\neq 0$ for at most one index $j\in\{1,\ldots,n\}$.

Without loss of generality we assume
$\alpha_{i}=0$ for $i\neq 1$ and $\beta_{j}=0$ for $j\neq 1$.
Note that $\alpha_{1}$ and $\beta_{1}$ are even and at least one of them is nonzero, say $\alpha_{1}\geq 2$.
Since
$$ |\mathcal{A}^{(\mathbf{p})}| \leq \bigg[
\vkwwww{\prod_{\substack{1\leq j\leq n\\ 1\leq \nu\leq b_j}}\!\! F_{1,j}(x_1,y_j^{(\nu)})}_{\!x_1}^2
\prod_{(i,\mu)\neq (1,1),(1,2)}\!\!
\vkqqqq{\prod_{\substack{1\leq j\leq n\\ 1\leq \nu\leq b_j}}\!\! F_{i,j}(x_i^{(\mu)}\!,y_j^{(\nu)})}_{x_i^{(\mu)}}
\bigg]_{\textrm{all }y_j^{(\nu)}} , $$
we can also assume $\alpha_1=2$, $\beta_1=0$.\,
Consider
$$ \mathcal{B}_{I\times J}^{(\mathbf{p})} :=
\vkqqqq{\prod_{\substack{1\leq i\leq m\\ 1\leq j\leq n}}
\,\prod_{\substack{1\leq \mu\leq a_i\\ 1\leq \nu\leq b_j}}\! F_{i,j}(x_i^{(\mu)}\!,y_j^{(\nu)})
}_{\substack{x_i^{(\mu)}\in I \textrm{ for all }(i,\mu)\\
y_j^{(\nu)}\in J \textrm{ for all }(j,\nu)}} . $$
Indeed, $\mathcal{B}^{(\mathbf{p})}$ depends only on $\mathbf{q}=\vkcomp(\mathbf{p})$.
Observe that Lemma \ref{vklemmahoelder} and Normalization (\ref{vkeqfnormalize}) guarantee
Condition (\ref{vkeqbcontrol}).
Theorem \ref{vktheoremterms} gives the equality
\begin{equation}
\label{vkeqboxequality}
\Box\mathcal{B}^{(\mathbf{p})}
= \!\sum_{\substack{\mathbf{p}'\in\Omega_{m,n}\\ \vkcomp(\mathbf{p}')=\mathbf{q}}}
\!\!\!{\mathbf{q}\choose \mathbf{p}'}\,\mathcal{A}^{(\mathbf{p}')} ,
\end{equation}
where
$$ {\mathbf{q}\choose \mathbf{p}'} := \prod_{i=1}^{m}\!{a_i\choose \alpha'_i} \,\prod_{j=1}^{n}\!{b_j\choose \beta'_j} $$
counts the repeating terms $\mathcal{A}^{(\mathbf{p}')}$.
Note that $1\leq {\mathbf{q}\choose \mathbf{p}'}\lesssim_{m,n} 1$.\,
Let us split the summation set
$$ \Omega_{m,n,\mathbf{q}} := \big\{\mathbf{p}'\in\Omega_{m,n} : \,\vkcomp(\mathbf{p}')=\mathbf{q}\big\}$$
into three parts,
\begin{align*}
\Omega_{m,n,\mathbf{q}}^{(1)} :=\, & \big\{\mathbf{p}'\in\Omega_{m,n,\mathbf{q}} :
\, \alpha'_{i}\neq 0 \textrm{ for exactly one } i \,\textrm{ and }\, \beta'_{j}=0 \textrm{ for every } j\big\} , \\
\Omega_{m,n,\mathbf{q}}^{(2)} :=\, & \big\{\mathbf{p}'\in\Omega_{m,n,\mathbf{q}} :
\, \beta'_{j}\neq 0 \textrm{ for exactly one } j\big\} , \\
\Omega_{m,n,\mathbf{q}}^{(3)} :=\, & \big\{\mathbf{p}'\in\Omega_{m,n,\mathbf{q}} :
\, \alpha'_{i}\neq 0 \textrm{ for at least two } i \,\textrm{ and }\, \beta'_{j}=0 \textrm{ for every } j\big\} \\
& \cup\big\{\mathbf{p}'\in\Omega_{m,n,\mathbf{q}} :
\, \beta'_{j}\neq 0 \textrm{ for at least two } j\big\} .
\end{align*}
First, observe that $\mathbf{p}\in\Omega^{(1)}_{m,n,\mathbf{q}}$ and that
$$ \mathcal{A}^{(\mathbf{p}')} = \bigg[
\vkwwww{\!\!\prod_{\substack{1\leq j\leq n\\ 1\leq \nu\leq b_j}}\!\! F_{i_0,j}(x_{i_0},y_j^{(\nu)})\!}_{\!x_{i_0}}^{2\sigma}
\prod_{\substack{(i,\mu)\textrm{ such that}\\ i\neq i_0 \textrm{ or } \mu>2\sigma}}\!\!
\vkqqqq{\!\prod_{\substack{1\leq j\leq n\\ 1\leq \nu\leq b_j}}\!\! F_{i,j}(x_i^{(\mu)}\!,y_j^{(\nu)})}_{x_i^{(\mu)}}
\bigg]_{\textrm{all }y_j^{(\nu)}} \!\geq 0 $$
for every $\mathbf{p}'\in\Omega^{(1)}_{m,n,\mathbf{q}}$, where $i_0$ is chosen such that $\alpha'_{i_0}=2\sigma\neq 0$.\,
In particular,
\begin{equation}
\label{vkeqsumcontrol1}
0 \leq \mathcal{A}^{(\mathbf{p})} \leq \!\sum_{\mathbf{p}'\in\Omega_{m,n,\mathbf{q}}^{(1)}}
\!\!\!{\mathbf{q}\choose \mathbf{p}'}\,\mathcal{A}^{(\mathbf{p}')} .
\end{equation}
Next, Lemma \ref{vklemmaterms} applied to
\begin{align*}
& \Psi = \Psi\big((x_i^{(\mu)})_{1\leq i\leq m,\ 1\leq\mu\leq a_i}\big) \\
& = \sum_{j_0=1}^{n}\!\sum_{\tau=1}^{\lfloor b_{j_0}/2\rfloor} \!{b_{j_0} \choose 2\tau}
\vkwwww{\!\!\prod_{\substack{1\leq i\leq m\\ 1\leq \mu\leq a_i}}\!\!
F_{i,j_0}(x_i^{(\mu)}\!,y_{j_0})\!}_{\!y_{j_0}}^{2\tau}
\prod_{\substack{(j,\nu)\textrm{ such that}\\ j\neq j_0 \textrm{ or } \nu>2\tau}}\!\!
\vkqqqq{\!\prod_{\substack{1\leq i\leq m\\ 1\leq \mu\leq a_i}}\!\!
F_{i,j}(x_i^{(\mu)}\!,y_j^{(\nu)})}_{y_j^{(\nu)}}\!\geq 0
\end{align*}
(where $\beta'_{j_0}=2\tau\neq 0$) yields
\begin{equation}
\label{vkeqsumcontrol2}
\sum_{\mathbf{p}'\in\Omega_{m,n,\mathbf{q}}^{(2)}}
\!\!\!{\mathbf{q}\choose \mathbf{p}'}\,\mathcal{A}^{(\mathbf{p}')} \geq 0 .
\end{equation}
Finally, each $\mathcal{A}^{(\mathbf{p}')}$ for
$\mathbf{p}'\in\Omega_{m,n,\mathbf{q}}^{(3)}$
can be controlled as in Case 1 to obtain
\begin{equation}
\label{vkeqsumcontrol3}
|\mathcal{A}^{(\mathbf{p}')}| \ \,\leq\
\delta^{-1}\!\!\!\!\!\!\!
\sum_{\substack{\widetilde{\mathbf{p}}\in\Omega_{m,n}\\ \vkord(\widetilde{\mathbf{p}})<\vkord(\mathbf{p})}}
\!\!\!\!\!\!\!|\mathcal{A}^{(\widetilde{\mathbf{p}})}|
\ +\ \delta\!\!\!\!\!\!
\sum_{\substack{\widetilde{\mathbf{p}}\in\Omega_{m,n}\\ \vkord(\widetilde{\mathbf{p}})\geq\vkord(\mathbf{p})}}
\!\!\!\!\!\!\!|\mathcal{A}^{(\widetilde{\mathbf{p}})}| .
\end{equation}
Combining (\ref{vkeqboxequality})--(\ref{vkeqsumcontrol3}) proves the stated inequality.\qedhere
\end{list}
\end{proof}

Figure \ref{vkfigurepartitiontree} depicts partition types in $\Omega^{\ast}_{2,4}$
and ways of controlling $\mathcal{A}^{(\mathbf{p})}$ as in Case 1 of the previous proof.
Different kinds of arrows represent different possibilities for various $\mathbf{p}\in\Omega_{2,4}$
with the same $\vkparti(\mathbf{p})$.
Labels $\delta$ and $\delta^{-1}$ represent coefficients in the ``reduction inequality'' for $\mathcal{A}^{(\mathbf{p})}$.
It is important that always at least one arrow (the one marked by $\delta^{-1}$) points
to a partition type with smaller rank, i.e.\@ points downwards or to the right in the picture.
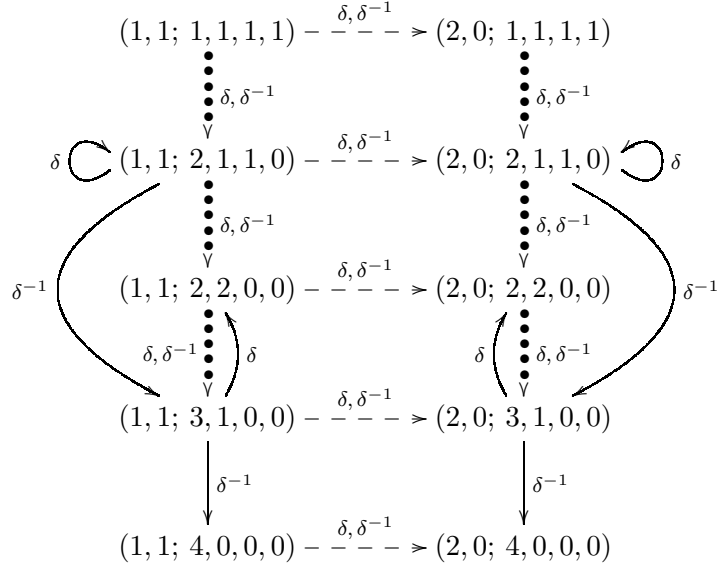
\begin{figure}[th]
{\small $$ \xymatrix @C=1.6cm @R=1.1cm {
(1,1;\,1,1,1,1)
\ar@{-->}[r]^{\delta,\,\delta^{-1}}
\ar@{{}{}{>}}[d]^{\,\delta,\,\delta^{-1}}
& (2,0;\,1,1,1,1)
\ar@{{}{}{>}}[d]^{\,\delta,\,\delta^{-1}} \\
(1,1;\,2,1,1,0)
\ar@{-->}[r]^{\delta,\,\delta^{-1}}
\ar@{{}{}{>}}[d]^{\,\delta,\,\delta^{-1}}
\ar@{->}@/_{2cm}/[dd]_{\delta^{-1}}
\save !L(0.9) \ar@{->}@(dl,ul)^{\delta} \restore
\save !<0ex,3.6ex> \ar@{{}{*}{}}[u] \restore
& (2,0;\,2,1,1,0)
\ar@{{}{}{>}}[d]^{\,\delta,\,\delta^{-1}} \ar@{->}@/^{2cm}/[dd]^{\delta^{-1}}
\save !R(0.9) \ar@{->}@(dr,ur)_{\delta} \restore
\save !<0ex,3.6ex> \ar@{{}{*}{}}[u] \restore \\
(1,1;\,2,2,0,0)
\ar@{-->}[r]^{\delta,\,\delta^{-1}}
\ar@{{}{}{>}}[d]_{\delta,\,\delta^{-1}}
\save !<0ex,3.6ex> \ar@{{}{*}{}}[u] \restore
& (2,0;\,2,2,0,0)
\ar@{{}{}{>}}[d]^{\,\delta,\,\delta^{-1}}
\save !<0ex,3.6ex> \ar@{{}{*}{}}[u] \restore \\
(1,1;\,3,1,0,0)
\ar@{-->}[r]^{\delta,\,\delta^{-1}}
\ar@{->}[d]^{\delta^{-1}}
\ar@{->}@/_{0.4cm}/[u]_{\delta}
\save !<0ex,3.6ex> \ar@{{}{*}{}}[u] \restore
& (2,0;\,3,1,0,0)
\ar@{->}[d]^{\delta^{-1}} \ar@{->}@/^{0.4cm}/[u]^{\delta}
\save !<0ex,3.6ex> \ar@{{}{*}{}}[u] \restore \\
(1,1;\,4,0,0,0)
\ar@{-->}[r]^{\delta,\,\delta^{-1}} & (2,0;\,4,0,0,0)
} $$ }
\vspace*{-4mm}
\caption{Recursive control of partition types in $\Omega^{\ast}_{2,4}$.}
\label{vkfigurepartitiontree}
\end{figure}

\begin{lemma}
\label{vklemmafinal}
\ There exists a ``universal'' averaging paraproduct-type expression\linebreak
$\mathcal{B}^{(m,n)}=\mathcal{B}^{(m,n)}_{I\times J}\big((F_{i,j})_{(i,j)\in E}\big)$
satisfying \emph{(\ref{vkeqbcontrol})} and
$$ \sum_{\mathbf{p}\in\Omega_{m,n}} \!\!|\mathcal{A}^{(\mathbf{p})}|
\ \leq \ \Box\mathcal{B}^{(m,n)} . $$
\end{lemma}
\begin{proof}[Proof of Lemma \ref{vklemmafinal}.]
We prove the following claim by mathematical induction on
$\kappa\in\{0,1,2,\ldots,\mathfrak{p}_{\#}(m)\mathfrak{p}_{\#}(n)\}$: \
\emph{For every $0<\varepsilon<1$ there exists an averaging paraproduct-type expression
$\mathcal{B}^{(\kappa,\varepsilon)}$ satisfying
$$ \max_{Q\in\mathcal{T}\cup\mathcal{L}(\mathcal{T})}\big|\mathcal{B}^{(\kappa,\varepsilon)}_{Q}
\big((F_{i,j})_{(i,j)\in E}\big)\big|\,\lesssim_{m,n,\varepsilon} 1 $$
and
$$ \sum_{\substack{\mathbf{p}\in\Omega_{m,n}\\ \vkord(\mathbf{p})\leq\kappa}}
\!\!|\mathcal{A}^{(\mathbf{p})}|
\ \,\leq \ \,\Box\mathcal{B}^{(\kappa,\varepsilon)}
\ +\ \varepsilon\!\!\sum_{\substack{\mathbf{p}\in\Omega_{m,n}\\ \vkord(\mathbf{p})>\kappa}}
\!\!|\mathcal{A}^{(\mathbf{p})}| . $$}
The bound we need to prove is a special instance for $\kappa=\mathfrak{p}_{\#}(m)\mathfrak{p}_{\#}(n)$ and any $\varepsilon$,
since then the sum on the right disappears.
On the other hand, the sum on the left is $0$ for $\kappa=0$, so the inequality (trivially) holds
with $\mathcal{B}^{(\kappa,\varepsilon)}\equiv 0$.
We take $\kappa=0$ as the induction basis.

In the induction step we suppose that the claim holds for some
$0\!\leq\!\kappa\!\leq\!\mathfrak{p}_{\!\#}(m)\mathfrak{p}_{\!\#}(n)-1$.
To prove the claim for $\kappa+1$ take arbitrary $0<\varepsilon<1$.
Let $C_{m,n}$ be the constant from Lemma \ref{vklemmareduction}.
Induction hypothesis applied with $\varepsilon'=\big(\frac{\varepsilon}{4 C_{m,n} |\Omega_{m,n}|}\big)^2$
gives $\mathcal{B}^{(\kappa,\varepsilon')}$ such that
\begin{equation}
\label{vkeqindhyp}
\sum_{\substack{\mathbf{p}\in\Omega_{m,n}\\ \vkord(\mathbf{p})\leq\kappa}}
\!\!|\mathcal{A}^{(\mathbf{p})}|
\ \,\leq \ \,\Box\mathcal{B}^{(\kappa,\varepsilon')}
\ +\ \big({\textstyle\frac{\varepsilon}{4 C_{m,n} |\Omega_{m,n}|}}\big)^2
\!\!\sum_{\substack{\mathbf{p}\in\Omega_{m,n}\\ \vkord(\mathbf{p})>\kappa}}
\!\!|\mathcal{A}^{(\mathbf{p})}| .
\end{equation}
Lemma \ref{vklemmareduction} applied to each $\mathbf{p}\in\Omega_{m,n}$, $\vkord(\mathbf{p})\leq\kappa+1$,
and $\delta=\frac{\varepsilon}{4 C_{m,n} |\Omega_{m,n}|}$ yields
$$ \sum_{\substack{\mathbf{p}\in\Omega_{m,n}\\ \vkord(\mathbf{p})\leq\kappa+1}}
\!\!\!\!|\mathcal{A}^{(\mathbf{p})}| \ \leq\
\sum_{\substack{\mathbf{p}\in\Omega_{m,n}\\ \vkord(\mathbf{p})\leq\kappa+1}}
\!\!\!\!\Box\mathcal{B}^{(\mathbf{p})}
\ +\ {\textstyle\frac{4 C_{m,n}^2 |\Omega_{m,n}|^2}{\varepsilon}}\!\!\sum_{\substack{\mathbf{p}\in\Omega_{m,n}\\
\vkord(\mathbf{p})\leq\kappa}}\!\!\!|\mathcal{A}^{(\mathbf{p})}|
\ +\ {\textstyle\frac{\varepsilon}{4}}\!\sum_{\mathbf{p}\in\Omega_{m,n}}\!\!|\mathcal{A}^{(\mathbf{p})}| , $$
which combined with (\ref{vkeqindhyp}) leads to
$$ \sum_{\substack{\mathbf{p}\in\Omega_{m,n}\\ \vkord(\mathbf{p})\leq\kappa+1}}
\!\!\!\!|\mathcal{A}^{(\mathbf{p})}| \ \leq\
\sum_{\substack{\mathbf{p}\in\Omega_{m,n}\\ \vkord(\mathbf{p})\leq\kappa+1}}
\!\!\!\!\Box\mathcal{B}^{(\mathbf{p})}
\ +\ {\textstyle\frac{4 C_{m,n}^2 |\Omega_{m,n}|^2}{\varepsilon}}\,\Box\mathcal{B}^{(\kappa,\varepsilon')}
\ +\ {\textstyle\frac{\varepsilon}{2}}\!\sum_{\mathbf{p}\in\Omega_{m,n}}\!\!|\mathcal{A}^{(\mathbf{p})}| . $$
It remains to split $\sum_{\mathbf{p}\in\Omega_{m,n}}$ on the right hand side
according to \,$\vkord(\mathbf{p})\leq\kappa+1$\, or \,$\vkord(\mathbf{p})>\kappa+1$,\,
move the former summands to the left hand side, and multiply the inequality by $2$. We set\footnote{From this
recurrence relation for $\mathcal{B}^{(\kappa,\varepsilon)}$ we see that the final Bellman function
is assembled as a linear combination of familiar terms $\mathcal{B}^{(\mathbf{p})}$,
but the coefficients can theoretically grow ``very rapidly'' and might be hard to track down explicitly in concrete examples.}
\begin{equation*}
\mathcal{B}^{(\kappa+1,\,\varepsilon)}
\ = \ 2\!\!\!\sum_{\substack{\mathbf{p}\in\Omega_{m,n}\\ \vkord(\mathbf{p})\leq\kappa+1}}\!\!\!\!\!\mathcal{B}^{(\mathbf{p})}
\ +\ {\textstyle\frac{8 C_{m,n}^2 |\Omega_{m,n}|^2}{\varepsilon}}\,\mathcal{B}^{(\kappa,\varepsilon')} . \tag*{\qedhere}
\end{equation*}
\end{proof}

Let us finally complete the proof of Proposition \ref{vkpropositiontree} in this special case.
Even though the general term $\mathcal{A}$ of $\Lambda$
is not necessarily among $\mathcal{A}^{(\mathbf{p})}$, $\mathbf{p}\in\Omega_{m,n}$,
it can still be easily dominated by two of these terms.
Remember that $|S|\geq 2$ or $|T|\geq 2$, so
without loss of generality suppose $|S|\geq 2$ and take $i_1,i_2\in S$,\, $i_1\neq i_2$.
\begin{align*}
|\mathcal{A}| & \leq \bigg[
\,\Big|\vkwwww{\prod_{1\leq j\leq n}\!\! F_{i_1,j}(x_{i_1},y_j)}_{\!x_{i_1}}
\!\vkwwww{\prod_{1\leq j\leq n}\!\! F_{i_2,j}(x_{i_2},y_j)}_{\!x_{i_2}}\Big|
\ \prod_{i\neq i_1,i_2}\!\vkqqqq{\prod_{1\leq j\leq n}\!\! F_{i,j}(x_i,y_j)}_{x_{i}}
\bigg]_{y_1,\ldots,y_n} \\
& \leq \ \frac{1}{2}\bigg[
\vkwwww{\prod_{1\leq j\leq n}\!\! F_{i_1,j}(x_{i_1},y_j)}_{\!x_{i_1}}^2
\prod_{i\neq i_1,i_2}\!\vkqqqq{\prod_{1\leq j\leq n}\!\! F_{i,j}(x_i,y_j)}_{x_{i}}
\bigg]_{y_1,\ldots,y_n} \\
& \ \,+ \frac{1}{2}\bigg[
\vkwwww{\prod_{1\leq j\leq n}\!\! F_{i_2,j}(x_{i_2},y_j)}_{\!x_{i_2}}^2
\prod_{i\neq i_1,i_2}\!\vkqqqq{\prod_{1\leq j\leq n}\!\! F_{i,j}(x_i,y_j)}_{x_{i}}
\bigg]_{y_1,\ldots,y_n}
= \ \frac{1}{2}\mathcal{A}^{(\widetilde{\mathbf{p}})}
+ \frac{1}{2}\mathcal{A}^{(\overline{\mathbf{p}})}
\end{align*}
Above $\widetilde{\mathbf{p}},\,\overline{\mathbf{p}}\in\Omega_{m,n}$ are given by their components:
$$ \begin{array}{l}\begin{array}{ll}
\widetilde{a}_i = \left\{\rule{0cm}{8mm}\right.\!\!\begin{array}{cl}
2, & \textrm{ for }i=i_1, \\ 0, & \textrm{ for }i=i_2, \\ 1, & \textrm{ for }i\neq i_1,i_2, \end{array}
& \overline{a}_i = \left\{\rule{0cm}{8mm}\right.\!\!\begin{array}{cl}
0, & \textrm{ for }i=i_1, \\ 2, & \textrm{ for }i=i_2, \\ 1, & \textrm{ for }i\neq i_1,i_2, \end{array} \\
\widetilde{\alpha}_i = \left\{\rule{0cm}{5.5mm}\right.\!\!\begin{array}{cl}
2, & \textrm{ for }i=i_1, \\ 0, & \textrm{ for }i\neq i_1, \end{array}
& \overline{\alpha}_i = \left\{\rule{0cm}{5.5mm}\right.\!\!\begin{array}{cl}
2, & \textrm{ for }i=i_2, \\ 0, & \textrm{ for }i\neq i_2, \end{array} \end{array} \\
\ \,\widetilde{b}_j = \overline{b}_j = 1,\ \,
\widetilde{\beta}_j = \overline{\beta}_j = 0\ \,\textrm{ for every }j .
\end{array} $$
Lemma \ref{vklemmafinal}, Bound (\ref{vkeqbcontrol}),
and the general Estimate (\ref{vkeqmainteleestimate})
establish (\ref{vkeqsingletreenormalized}).

\subsection{Proof of Proposition \ref{vkpropositiontree} in full generality}
\label{vksubsectionfull}

Let us now consider a general triple $(E,S,T)$.
Suppose that there are $k$ connected components that contain at least one edge and list their vertex-sets as
\,$X_1\cup Y_1$, $X_2\cup Y_2$, \ldots, $X_k\cup Y_k$,\,
where $X_l\subseteq\{x_1,\ldots,x_m\}$ and $Y_l\subseteq\{y_1,\ldots,y_n\}$ for $l=1,\ldots,k$.
The numbers $d_{i,j}$ were defined to be
$$ d_{i,j} = d^{(l)} := \max\{|X_l|,|Y_l|\big\} $$
if \,$x_i\in X_l$,\, $y_j\in Y_l$\, for some $l\in\{1,2,\ldots,k\}$.\,
Since $d_{i,j}$ depends on the component only, we sometimes prefer to write it as $d^{(l)}$.
For notational convenience we also denote
$$ \mathcal{X}_l := \big\{i\in\{1,\ldots,m\} \,:\, x_i\in X_l\big\}, \quad\
\mathcal{Y}_l := \big\{j\in\{1,\ldots,n\} \,:\, y_j\in Y_l\big\}, $$
for $l=1,\ldots,k$.

Vertices in $\{x_1,\ldots,x_m\}\setminus(X_1\cup\ldots\cup X_k)$ and
$\{y_1,\ldots,y_n\}\setminus(Y_1\cup\ldots\cup Y_k)$ are \emph{isolated},
i.e.\@ no edge is incident to any of them.
Note that isolated vertices contribute to the form $\Lambda$ in a trivial way.
If $x_i$ is a non-selected isolated vertex, then no functions $F_{i,j}$ will have $x_i$ as a variable,
so the contribution of $x_i$ is only in the factor
$$ |I|^{-\frac{1}{2}}\int_{\mathbb{R}}\varphi^{\mathrm{D}}_I(x_i) dx_i = 1 . $$
On the other hand, if $x_i$ is a selected isolated vertex, then that factor will be
$$ |I|^{-\frac{1}{2}}\int_{\mathbb{R}}\psi^{\mathrm{D}}_I(x_i) dx_i = 0 , $$
so consequently $\Lambda\equiv 0$.
The same reasoning holds for isolated vertices $y_j$.
Therefore, from now on we assume that the associated bipartite graph has no isolated vertices.

By taking $F_{i,j}\equiv\mathbf{1}$ whenever $(i,j)\not\in E$ and
``completing'' each component of the graph, we can assume
$$ E = (\mathcal{X}_1 \times \mathcal{Y}_1) \cup \ldots \cup (\mathcal{X}_k \times \mathcal{Y}_k) . $$
Note that the completed graph still has the same characteristic quantities $d_{i,j}$
and that functions $F_{i,j}$ identically equal to $1$ do not contribute to
either of the sides in (\ref{vkeqsingletree}).
Connected components are useful because the associated forms $\Lambda_{\mathcal{T}}$ split as
$$ \Lambda_{\mathcal{T}}\big((F_{i,j})_{(i,j)\in E}\big)
:= \sum_{Q\in\mathcal{T}} |Q| \,\prod_{l=1}^{k}\big|\mathcal{A}_{Q}^{(l)}
\big((F_{i,j})_{(i,j)\in\mathcal{X}_l\times\mathcal{Y}_l}\big)\big|, $$
with
$$ \mathcal{A}_{I\times J}^{(l)} =
\vkqqqq{\vkwwww{\!\prod_{i\in\mathcal{X}_l,\, j\in\mathcal{Y}_l}\!\!\! F_{i,j}(x_i,y_j)
}_{\substack{x_i\in I \textrm{ for } i\in S\cap\mathcal{X}_l\\ y_j\in J \textrm{ for } j\in T\cap\mathcal{Y}_l}\,}
}_{\substack{x_i\in I \textrm{ for } i\in S^c\cap\mathcal{X}_l\\ y_j\in J \textrm{ for } j\in T^c\cap\mathcal{Y}_l}} . $$

We distinguish two cases with respect to the distribution of selected vertices in the graph.
Recall once again that $|S|\geq 2$ or $|T|\geq 2$, which guarantees that the following two cases
cover all possibilities, although not necessarily disjoint ones.

\smallskip
\begin{list}{\labelitemi}{\setlength{\leftmargin}{0em}
\setlength{\itemsep}{5pt}\setlength{\itemindent}{0pt}}
\item[]\emph{Case 1.} \
$|S\cap\mathcal{X}_{l_0}|\geq 2$ \,or\, $|T\cap\mathcal{Y}_{l_0}|\geq 2$ \,for some index $l_0\in\{1,\ldots,k\}$.

In words, some connected component contains two selected vertices
in one of its bipartition classes.
From the previously proven case of Proposition \ref{vkpropositiontree} applied to
the complete bipartite graph with vertex-sets $X_{l_0}$,\,$Y_{l_0}$ we have the estimate
$$ \sum_{Q\in\mathcal{T}} |Q| \,\big|\mathcal{A}_{Q}^{(l_0)}
\big((F_{i,j})_{(i,j)\in\mathcal{X}_{l_0}\times\mathcal{Y}_{l_0}}\big)\big|
\,\lesssim_{m,n} |Q_\mathcal{T}|\!\prod_{i\in\mathcal{X}_{l_0}\!,\, j\in\mathcal{Y}_{l_0}}\!
\max_{Q\in\mathcal{T}\cup\mathcal{L}(\mathcal{T})}\vkqqq{F_{i,j}^{d^{(l_0)}}}_{Q}^{1/d^{(l_0)}} = 1 . $$
Applying Lemma \ref{vklemmahoelder} to the functions
$(F_{i,j})_{i\in\mathcal{X}_l,\,j\in\mathcal{Y}_l}$ we get
$$ \big|\mathcal{A}_{Q}^{(l)}\big((F_{i,j})_{(i,j)\in\mathcal{X}_l\times\mathcal{Y}_l}\big)\big|
\leq\vkqqqq{\!\prod_{i\in\mathcal{X}_l,\, j\in\mathcal{Y}_l}\!\!\! F_{i,j}(x_i,y_j)
}_{\substack{x_i\in I \textrm{ for } i\in\mathcal{X}_l\\ y_j\in J \textrm{ for } j\in\mathcal{Y}_l}}
\leq \!\prod_{i\in\mathcal{X}_l,\, j\in\mathcal{Y}_l}\!\!\! \vkqqq{F_{i,j}^{d^{(l)}}}_{Q}^{1/d^{(l)}} \leq 1 $$
for each $Q=I\times J\in\mathcal{T}$ and each $l\neq l_0$.
This establishes (\ref{vkeqsingletreenormalized}).
\pagebreak

\item[]\emph{Case 2.} \
$S\cap\mathcal{X}_{l_1}\neq\emptyset\neq S\cap\mathcal{X}_{l_2}$ \,or\,
$T\cap\mathcal{Y}_{l_1}\neq\emptyset\neq T\cap\mathcal{Y}_{l_2}$
\,for two different indices $l_1,l_2\in\{1,\ldots,k\}$.

In words, there exist two connected components each containing at least one selected $x$-vertex
or each containing at least one selected $y$-vertex.
Without loss of generality suppose
$S\cap\mathcal{X}_{1}\neq 0$ and $S\cap\mathcal{X}_{2}\neq 0$.
Using Lemma \ref{vklemmahoelder} we can estimate the general term as
$$ |\mathcal{A}| \,=\, |\mathcal{A}^{(1)}| |\mathcal{A}^{(2)}| |\mathcal{A}^{(3)}| \ldots |\mathcal{A}^{(k)}|
\,\leq\, |\mathcal{A}^{(1)}| |\mathcal{A}^{(2)}|
\,\leq\, \frac{1}{2}\big(\mathcal{A}^{(1)}\big)^2 + \frac{1}{2}\big(\mathcal{A}^{(2)}\big)^2 . $$
By symmetry it is enough to handle $(\mathcal{A}^{(1)})^2$
and by renaming vertices we may assume
$\mathcal{X}_1=\{1,\ldots,m_1\}$, $\mathcal{Y}_1=\{1,\ldots,n_1\}$, and $1\in S$.

\smallskip
If $m_1=n_1=1$, then from the second row of Table \ref{vktableexamples},
\begin{align*}
& (\mathcal{A}^{(1)})^2
= \vkq{\vkw{F_{1,1}(x_1,y_1)}_{x_1}}_{y_1}^2 \textrm{ or } \vkw{F_{1,1}(x_1,y_1)}_{x_1,y_1}^2 \\[-5mm]
& \leq \vkq{\vkw{F_{1,1}(x_1,y_1)}_{x_1}}_{y_1}^2 + \vkw{F_{1,1}(x_1,y_1)}_{x_1,y_1}^2
+ \vkw{\vkq{F_{1,1}(x_1,y_1)}_{x_1}}_{y_1}^2
= \Box\big(\overbrace{\vkq{F_{1,1}(x_1,y_1)}_{x_1,y_1}^2}^{\mathcal{B}}\big) .
\end{align*}
We separate this special case because now $d^{(1)}=1$
(as for classical paraproducts), so the normalization (\ref{vkeqfnormalize})
does not control averages of higher powers of $F_{1,1}$.
However, the first power is enough here:
$$ \max_{Q\in\mathcal{T}\cup\mathcal{L}(\mathcal{T})} \mathcal{B}_{Q}(F_{1,1})
\,=\, \Big(\max_{Q\in\mathcal{T}\cup\mathcal{L}(\mathcal{T})}\vkq{F_{1,1}}_{Q}\Big)^2 = 1 . $$

On the other hand, the condition $d^{(1)}\geq 2$ is ensured if \,$m_1=1$, $n_1\geq 2$,\,
so then we bound $(\mathcal{A}^{(1)})^2$ as
$$ (\mathcal{A}^{(1)})^2 \,\leq\, \bigg[
\,\Big|\vkwwww{\prod_{j=1}^{n_1} F_{1,j}(x_1,y_j)}_{\!x_1}\Big|\,
\bigg]_{y_1,\ldots,y_{n_1}}^2
\!\!\leq \,\bigg[
\vkwwww{\prod_{j=1}^{n_1} F_{1,j}(x_1,y_j)}_{\!x_1}^2
\bigg]_{y_1,\ldots,y_{n_1}} . $$
We can recognize the last term as $\mathcal{A}^{(\widetilde{\mathbf{p}})}$
for $\widetilde{\mathbf{p}}\in\Omega_{2,n_1}$ given by
\ $\widetilde{a}_1 = \widetilde{\alpha}_1 = 2$, \ $\widetilde{a}_2 = \widetilde{\alpha}_2 = 0$, \
$\widetilde{b}_j = 1,\,\widetilde{\beta}_j = 0\,$ for $1\leq j\leq n_1$.
Thus,
\,$(\mathcal{A}^{(1)})^2 \leq \mathcal{A}^{(\widetilde{\mathbf{p}})} \leq \Box\mathcal{B}^{(2,n_1)}$,\,
where $\mathcal{B}^{(2,n_1)}$ is from Lemma \ref{vklemmafinal}.

If $m_1\geq 2$, then one can write
with the help of the Cauchy-Schwarz inequality and Lemma \ref{vklemmahoelder} once again:
\begin{align*}
(\mathcal{A}^{(1)})^2 \ & \leq \ \bigg[
\,\Big|\vkwwww{\prod_{j=1}^{n_1} F_{1,j}(x_1,y_j)}_{\!x_1}\Big|
\ \,\vkqqqq{\prod_{j=1}^{n_1} F_{2,j}(x_2,y_j)}_{\!x_2}
\ \prod_{i=3}^{m_1}\vkqqqq{\prod_{j=1}^{n_1} F_{i,j}(x_i,y_j)}_{x_{i}}
\bigg]_{y_1,\ldots,y_{n_1}}^2 \\
& \leq \ \ \underbrace{\bigg[
\vkwwww{\prod_{j=1}^{n_1} F_{1,j}(x_1,y_j)}_{\!x_1}^2
\ \prod_{i=3}^{m_1}\vkqqqq{\prod_{j=1}^{n_1} F_{i,j}(x_i,y_j)}_{x_{i}}
\bigg]_{y_1,\ldots,y_{n_1}}}_{\qquad\qquad\mathcal{A}^{(\overline{\mathbf{p}})}\leq\Box\mathcal{B}^{(m_1,n_1)}} \\
& \ \ \,\times \underbrace{\bigg[
\vkqqqq{\prod_{j=1}^{n_1} F_{2,j}(x_2,y_j)}_{x_2}^2
\ \prod_{i=3}^{m_1}\vkqqqq{\prod_{j=1}^{n_1} F_{i,j}(x_i,y_j)}_{x_{i}}
\bigg]_{y_1,\ldots,y_{n_1}}}_{\ \leq 1} ,
\end{align*}
where $\overline{\mathbf{p}}\in\Omega_{m_1,n_1}$ has coordinates
$$ \begin{array}{l}\begin{array}{ll}
\overline{a}_i = \left\{\rule{0cm}{8mm}\right.\!\!\begin{array}{cl}
2, & \textrm{ for }i=1, \\ 0, & \textrm{ for }i=2, \\ 1, & \textrm{ for }3\leq i\leq m_1, \end{array}
& \overline{\alpha}_i = \left\{\rule{0cm}{5.5mm}\right.\!\!\begin{array}{cl}
2, & \textrm{ for }i=1, \\ 0, & \textrm{ for }2\leq i\leq m_1, \end{array} \end{array} \\
\ \,\overline{b}_j = 1,\ \,
\overline{\beta}_j = 0\ \,\textrm{ for }1\leq j\leq n_1
\end{array} $$
and $\mathcal{B}^{(m_1,n_1)}$ is from Lemma \ref{vklemmafinal}.

In all three possibilities above the proof is finished by invoking
Estimate (\ref{vkeqmainteleestimate}).
\end{list}

\section{Completing the proof of the main theorem}
\label{vksec4maintheorem}

To establish the global estimate we adapt the approach by Thiele from \cite{T1}.
\begin{proof}[Proof of Theorem \ref{vktheoremmainbound}.]
Fix a positive integer $N$ and consider only squares with sidelength at least $2^{-N}$,
$$ \mathcal{C}_{N} := \Big\{Q=I\times J\in \mathcal{C} \,:\, |I|=|J|\geq 2^{-N} \Big\} . $$
We prove the bound
\begin{equation}
\label{vkeqthmproof0}
\sum_{Q\in\mathcal{C}_N} \!|Q| \,\big|\mathcal{A}_{Q}\big((F_{i,j})_{(i,j)\in E}\big)\big|
\ \lesssim_{m,n,(p_{i,j})} \prod_{(i,j)\in E}\!\! \|F_{i,j}\|_{\mathrm{L}^{p_{i,j}}(\mathbb{R}^2)} ,
\end{equation}
with the implicit constant independent of $N$, so that it implies the result for the whole collection $\mathcal{C}$.
Using homogeneity, this time we normalize
$$ \|F_{i,j}\|_{\mathrm{L}^{p_{i,j}}(\mathbb{R}^2)}=1\quad \textrm{for every $(i,j)\in E$}. $$

For each $|E|$-tuple of integers $\mathbf{k}=(k_{i,j})_{(i,j)\in E}\in\mathbb{Z}^{|E|}$ we denote
$$ \mathcal{P}_{\mathbf{k}} := \Big\{Q\in\mathcal{C}_{N} \,:\, 2^{k_{i,j}}\leq
\sup_{Q'\in\mathcal{C}_{N},\ Q'\supseteq Q}\!
\vkqqq{F_{i,j}^{d_{i,j}}}_{Q'}^{1/d_{i,j}} <\,2^{k_{i,j}+1} \,\textrm{ for every } (i,j)\in E \Big\} . $$
Note that squares in $\mathcal{P}_{\mathbf{k}}$ satisfy
\,$|Q|\leq 2^{-p_{i,j} (k_{i,j}-1)}$\, for any $(i,j)\in E$,
which limits their size from above.
To verify this, we take $Q\in\mathcal{P}_{\mathbf{k}}$ and choose $Q'\supseteq Q$ such that
$\vkqqq{F_{i,j}^{d_{i,j}}}_{Q'}^{1/d_{i,j}} > 2^{k_{i,j}-1}$.
By the monotonicity of normalized $\mathrm{L}^p$ norms
$$ 2^{k_{i,j}-1} < \vkqqq{F_{i,j}^{d_{i,j}}}_{Q'}^{1/{d_{i,j}}} \leq \vkqqq{F_{i,j}^{p_{i,j}}}_{Q'}^{1/{p_{i,j}}}
= |Q'|^{-1/p_{i,j}}\|F_{i,j}\|_{\mathrm{L}^{p_{i,j}}(Q')} \leq |Q'|^{-1/p_{i,j}} $$
and thus $|Q|\leq |Q'|\leq 2^{-p_{i,j} (k_{i,j}-1)}$.

Define $\mathcal{M}_{\mathbf{k}}$ to be the collection of maximal squares in $\mathcal{P}_{\mathbf{k}}$
with respect to the set inclusion.
For each $Q\in\mathcal{M}_{\mathbf{k}}$ the family
$$ \mathcal{T}_Q := \big\{ \widetilde{Q}\in \mathcal{P}_{\mathbf{k}} \, : \, \widetilde{Q}\subseteq Q \big\} $$
is a finite convex\footnote{Convexity of $\mathcal{T}_Q$ follows from monotonicity
of $\widetilde{Q}\mapsto\sup_{Q'\in\mathcal{C}_{N},\, Q'\supseteq\widetilde{Q}}$.}
tree with root $Q$ and for different squares $Q\in\mathcal{M}_{\mathbf{k}}$ the corresponding trees
cover disjoint regions in the plane.
For $\widetilde{Q}\in\mathcal{T}_Q$ by the construction of $\mathcal{P}_{\mathbf{k}}$ we have
$\vkqqq{F_{i,j}^{d_{i,j}}}_{\widetilde{Q}}^{1/d_{i,j}} < 2^{k_{i,j}+1}$.
Also, if $\widetilde{Q}\in\mathcal{L}(\mathcal{T}_Q)$ and $\overline{Q}$ is the parent of $\widetilde{Q}$, then
$\vkqqq{F_{i,j}^{d_{i,j}}}_{\widetilde{Q}}^{1/d_{i,j}} \leq
\ 4\vkqqq{F_{i,j}^{d_{i,j}}}_{\overline{Q}}^{1/d_{i,j}} < 2^{k_{i,j}+3}$.
Therefore, Proposition \ref{vkpropositiontree} gives
$$ \Lambda_{\mathcal{T}_Q}\big((F_{i,j})_{(i,j)\in E}\big)
\ \lesssim_{m,n} |Q|\ 2^{\sum_{(i,j)\in E} k_{i,j}} . $$
We decompose\footnote{Here we use that for each
$\widetilde{Q}\in\mathcal{C}_{N}\setminus\bigcup_{\mathbf{k}\in\mathbb{Z}^{|E|}}\mathcal{P}_{\mathbf{k}}$
at least one of the functions constantly vanishes on $\widetilde{Q}$,
so the corresponding summand is equal to $0$.}
and estimate,
\begin{align}
\sum_{Q\in\mathcal{C}_N} \!|Q| \,\big|\mathcal{A}_{Q}\big((F_{i,j})_{(i,j)\in E}\big)\big|\
& = \ \sum_{\mathbf{k}\in\mathbb{Z}^{|E|}}\,\sum_{Q\in\mathcal{M}_{\mathbf{k}}}
\Lambda_{\mathcal{T}_Q}\big((F_{i,j})_{(i,j)\in E}\big) \nonumber \\
\label{vkeqthmproof1}
& \lesssim_{m,n} \sum_{\mathbf{k}\in\mathbb{Z}^{|E|}} 2^{\sum_{(i,j)\in E} k_{i,j}}
\Big(\sum_{Q\in\mathcal{M}_{\mathbf{k}}}\!\! |Q|\Big) .
\end{align}

For any $d\geq 1$ it is a classical result\footnote{One simply writes
$\mathrm{M}_{d} F = \big(\mathrm{M}_{1} |F|^d\big)^{1/d}$, where
$\mathrm{M}_{1}$ is the standard dyadic maximal function.}
that ``power $d$ variant'' of the \emph{dyadic maximal function}
$$ (\mathrm{M}_{d} F)(x,y) := \!\sup_{Q\in\mathcal{C},\ Q\ni (x,y)}
\!\vkqqq{|F|^d}_{Q}^{1/d} $$
is bounded on $\mathrm{L}^p(\mathbb{R}^2)$ whenever $d<p\leq\infty$.
We have $\mathbb{Z}^{|E|} =\bigcup_{(i_0,j_0)\in E}\mathcal{K}_{(i_0,j_0)}$,
where the subsets $\mathcal{K}_{(i_0,j_0)}$ are defined by
$$ \mathcal{K}_{(i_0,j_0)} := \big\{ \mathbf{k}=(k_{i,j})_{(i,j)\in E} \,:\,
p_{i_0,j_0}k_{i_0,j_0}\geq p_{i,j}k_{i,j}\textrm{ for every }(i,j)\in E \big\} . $$
Observe that for $(x,y)\in Q\in\mathcal{P}_{\mathbf{k}}$ we have by the definition of $\mathcal{P}_{\mathbf{k}}$,
$$ (\mathrm{M}_{d_{i_0,j_0}}\!F_{i_0,j_0})(x,y) \geq \sup_{Q'\supseteq Q,\ Q'\in\mathcal{C}_{N}}\!
\vkqqq{F_{i_0,j_0}^{d_{i_0,j_0}}}_{Q'}^{1/d_{i_0,j_0}} \geq 2^{k_{i_0,j_0}} , $$
which together with disjointness of $\mathcal{M}_{\mathbf{k}}$ yields
\begin{equation}
\label{vkeqthmproof2}
\sum_{Q\in\mathcal{M}_{\mathbf{k}}}\!\! |Q| \,=\, \Big|\!\bigcup_{Q\in\mathcal{M}_{\mathbf{k}}}\!\! Q\Big|
\,\leq\,\big|\big\{(x,y)\in\mathbb{R}^2 : (\mathrm{M}_{d_{i_0,j_0}}\!F_{i_0,j_0})(x,y) \geq 2^{k_{i_0,j_0}}\big\}\big|
\end{equation}
for any choice of $(i_0,j_0)\in E$.

Combining (\ref{vkeqthmproof1}) and (\ref{vkeqthmproof2}) allows the final computation:
\begin{align*}
& \sum_{Q\in\mathcal{C}_N}\! |Q| \,\big|\mathcal{A}_{Q}\big((F_{i,j})_{(i,j)\in E}\big)\big|
\ \lesssim_{m,n} \sum_{(i_0,j_0)\in E} \ \sum_{\mathbf{k}\in\mathcal{K}_{(i_0,j_0)}} 2^{\sum_{(i,j)\in E} k_{i,j}}
\ \big|\big\{\mathrm{M}_{d_{i_0,j_0}}\!F_{i_0,j_0} \geq 2^{k_{i_0,j_0}}\big\}\big| \\
& \ = \sum_{(i_0,j_0)\in E} \, \sum_{k_{i_0,j_0}\in\mathbb{Z}}
2^{p_{i_0,j_0}k_{i_0,j_0}}\,\big|\big\{\mathrm{M}_{d_{i_0,j_0}}\!F_{i_0,j_0} \geq 2^{k_{i_0,j_0}}\big\}\big|
\prod_{\substack{(i,j)\in E\\ (i,j)\neq (i_0,j_0)}}
\ \sum_{\substack{k_{i,j}\in\mathbb{Z}\\ k_{i,j}\leq\frac{p_{i_0,j_0}k_{i_0,j_0}}{p_{i,j}}}}
\!\!\!\!\!\!\!\! 2^{\,k_{i,j}-\frac{p_{i_0,j_0}k_{i_0,j_0}}{p_{i,j}}} \\[-2mm]
& \ \lesssim_{|E|} \sum_{(i_0,j_0)\in E} \, \sum_{k_{i_0,j_0}\in\mathbb{Z}}
2^{p_{i_0,j_0}k_{i_0,j_0}}\,\big|\big\{\mathrm{M}_{d_{i_0,j_0}}\!F_{i_0,j_0} \geq 2^{k_{i_0,j_0}}\big\}\big| \\
& \ \lesssim_{(p_{i,j})} \sum_{(i_0,j_0)\in E}
\|\mathrm{M}_{d_{i_0,j_0}}\!F_{i_0,j_0}\|_{\mathrm{L}^{p_{i_0,j_0}}(\mathbb{R}^2)}^{p_{i_0,j_0}}
\ \lesssim_{(d_{i,j}),\,(p_{i,j})} \sum_{(i_0,j_0)\in E}
\|F_{i_0,j_0}\|_{\mathrm{L}^{p_{i_0,j_0}}(\mathbb{R}^2)}^{p_{i_0,j_0}}
\ \lesssim_{|E|} \ 1 ,
\end{align*}
which is exactly (\ref{vkeqthmproof0}).
We used \,$\sum_{(i,j)\in E}\frac{1}{p_{i,j}}=1$\, and
added up $|E|\!-\!1$ geometric series with initial terms in $(\frac{1}{2},1]$ and common ratios
equal to $\frac{1}{2}$.
\end{proof}

\section{An example that illustrates the proof}
\label{vksec5example}

In this short section we show how important steps of the proof
should be performed on a concrete example given in Figure \ref{vkfigureexample}.
\vspace*{-3mm}
\begin{figure}[th]
\begin{center}\includegraphics[width=0.362\textwidth]{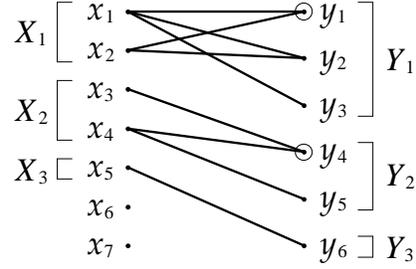}\end{center}
\vspace*{-4mm}
\caption{An example of a bipartite graph and its splitting.}
\label{vkfigureexample}
\end{figure}

\noindent
The example has parameters
\begin{align*}
& m=7,\ n=6,\ E=\big\{(1,1),(1,2),(1,3),(2,1),(2,2),(3,4),(4,4),(4,5),(5,6)\big\}, \\
& S=\emptyset,\ T=\{1,4\},\ d_{1,1}\!=\!d_{1,2}\!=\!d_{1,3}\!=\!d_{2,1}\!=\!d_{2,2}\!=\!3,\
d_{3,4}\!=\!d_{4,4}\!=\!d_{4,5}\!=\!2,\ d_{5,6}\!=\!1 ;
\end{align*}
the splitting into connected components is
$$ \mathcal{X}_1=\{1,2\},\ \mathcal{Y}_1=\{1,2,3\},\ \mathcal{X}_2=\{3,4\},\ \mathcal{Y}_2=\{4,5\},\
\mathcal{X}_3=\{5\},\ \mathcal{Y}_3=\{6\} ; $$
and the forms we associate to this graph decompose as
$$ \Lambda_\mathcal{T} = \sum_{Q\in\mathcal{T}} \,|Q|\,\big|\mathcal{A}_{Q}^{(1)}\big|
\,\big|\mathcal{A}_{Q}^{(2)}\big|\,\big|\mathcal{A}_{Q}^{(3)}\big| , $$
where
\begin{align*}
\mathcal{A}^{(1)} & = \vkqqq{\vkwww{F_{1,1}(x_1,y_1)F_{1,2}(x_1,y_2)F_{1,3}(x_1,y_3)F_{2,1}(x_2,y_1)F_{2,2}(x_2,y_2)}_{y_1}}_{x_1,x_2,y_2,y_3}, \\
\mathcal{A}^{(2)} & = \vkqqq{\vkwww{F_{3,4}(x_3,y_4)F_{4,4}(x_4,y_4)F_{4,5}(x_4,y_5)}_{y_4}}_{x_3,x_4,y_5}, \\
\mathcal{A}^{(3)} & = \vkqqq{F_{5,6}(x_5,y_6)}_{x_5,y_6}.
\end{align*}

In order to deal with more symmetric situation, we ``complete'' the graph
by introducing $F_{2,3}\equiv\mathbf{1}$ and $F_{3,5}\equiv\mathbf{1}$.
Since $|T|=2$, we decide to rewrite
{\small\begin{align*}
\mathcal{A}^{(1)} & = \vkqqqq{
\vkwww{F_{1,1}(x_1,y_1)F_{2,1}(x_2,y_1)}_{y_1}\vkqqq{F_{1,2}(x_1,y_2)F_{2,2}(x_2,y_2)}_{y_2}
\vkqqq{F_{1,3}(x_1,y_3)F_{2,3}(x_2,y_3)}_{y_3}
}_{x_1,x_2}, \\
\mathcal{A}^{(2)} & = \vkqqqq{
\vkwww{F_{3,4}(x_3,y_4)F_{4,4}(x_4,y_4)}_{y_4}
\vkqqq{F_{3,5}(x_3,y_5)F_{4,5}(x_4,y_5)}_{y_5}}_{x_3,x_4}
\end{align*}}
and start by estimating the general term as in Subsection \ref{vksubsectionfull}, Case 2,
{\small\begin{align*}
|\mathcal{A}| & = |\mathcal{A}^{(1)}|\, |\mathcal{A}^{(2)}|\, |\mathcal{A}^{(3)}|
\leq \frac{1}{2}\big(\mathcal{A}^{(1)}\big)^2 + \frac{1}{2}\big(\mathcal{A}^{(2)}\big)^2 \\
& \leq\ \frac{1}{2}\underbrace{\vkqqqq{\vkwww{F_{1,1}(x_1,y_1)F_{2,1}(x_2,y_1)}_{y_1}^2
\vkqqq{F_{1,3}(x_1,y_3)F_{2,3}(x_2,y_3)}_{y_3}}_{x_1,x_2}
}_{\ \mathcal{A}^{(4)}} \\
& \qquad\times\underbrace{\vkqqqq{\vkqqq{F_{1,2}(x_1,y_2)F_{2,2}(x_2,y_2)}_{y_2}^2
\vkqqq{F_{1,3}(x_1,y_3)F_{2,3}(x_2,y_3)}_{y_3}}_{x_1,x_2}
}_{\ \leq 1} \\
& \ \,+ \frac{1}{2}\underbrace{\vkqqqq{\vkwww{F_{3,4}(x_3,y_4)F_{4,4}(x_4,y_4)}_{y_4}^2}_{x_3,x_4}
}_{\ \mathcal{A}^{(5)}}
\ \underbrace{\vkqqqq{\vkqqq{F_{3,5}(x_3,y_5)F_{4,5}(x_4,y_5)}_{y_5}^2}_{x_3,x_4}}_{\ \leq 1}.
\end{align*}}
Let us proceed with $\mathcal{A}^{(4)}$, since $\mathcal{A}^{(5)}$
is in the third row of Table \ref{vktableexamples} and can be treated even more easily.
Following the approach from the proof of Lemma \ref{vklemmareduction}, Case 2, we define
{\small\begin{align*}
\mathcal{B}^{(4)} & = \vkqqqq{\vkqqq{F_{1,1}(x_1,y_1)F_{2,1}(x_2,y_1)}_{y_1}^2
\vkqqq{F_{1,3}(x_1,y_3)F_{2,3}(x_2,y_3)}_{y_3}}_{x_1,x_2} \\
& = \vkqqqq{\vkqqq{F_{1,1}(x_1,y_1)F_{1,1}(x_1,y'_1)F_{1,3}(x_1,y_3)}_{x_1}
\vkqqq{F_{2,1}(x_2,y_1)F_{2,1}(x_2,y'_1)F_{2,3}(x_2,y_3)}_{x_2}}_{y_1,y'_1,y_3}.
\end{align*}}
By Theorem \ref{vktheoremterms} every term in $\Box\mathcal{B}^{(4)}$ is either $\mathcal{A}^{(4)}$, or
{\small\begin{align*}
\mathcal{A}^{(6)} & \!= \vkqqqq{\vkwww{F_{1,1}(x_1,y_1)F_{2,1}(x_2,y_1)}_{y_1}
\vkqqq{F_{1,1}(x_1,y_1)F_{2,1}(x_2,y_1)}_{y_1}
\vkwww{F_{1,3}(x_1,y_3)F_{2,3}(x_2,y_3)}_{\!y_3}}_{x_1,x_2},
\end{align*}}
or is of the form
{\small\begin{align*}
\mathcal{A}^{(7)} & = \Big(\vkwww{F_{1,1}(x_1,y_1)F_{1,1}(x_1,y'_1)F_{1,3}(x_1,y_3)}_{x_1}
\vkwww{F_{2,1}(x_2,y_1)F_{2,1}(x_2,y'_1)F_{2,3}(x_2,y_3)}_{x_2}\Big)_{y_1,y'_1,y_3},
\end{align*}}
where parentheses $(\cdot)$ are understood as either $[\cdot]$ or $\langle\cdot\rangle$ for each of the variables
$y_1$,\,$y'_1$,\,$y_3$ independently.
We dominate $\mathcal{A}^{(6)}$, $\mathcal{A}^{(7)}$ as in the proof of Lemma \ref{vklemmareduction}, Case 1.
{\small\begin{align*}
|\mathcal{A}^{(6)}| & \leq \frac{1}{2\delta}\underbrace{\vkqqqq{\vkwww{F_{1,1}(x_1,y_1)F_{2,1}(x_2,y_1)}_{y_1}^2
\vkqqq{F_{1,1}(x_1,y_1)F_{2,1}(x_2,y_1)}_{y_1}}_{x_1,x_2}}_{\mathcal{A}^{(8)}} \\
& \ \ + \frac{\delta}{2}\underbrace{\vkqqqq{\vkwww{F_{1,3}(x_1,y_3)F_{2,3}(x_2,y_3)}_{y_3}^2
\vkqqq{F_{1,1}(x_1,y_1)F_{2,1}(x_2,y_1)}_{y_1}}_{x_1,x_2}}_{\mathcal{A}^{(9)}} \\
|\mathcal{A}^{(7)}| & \leq \ \,\frac{1}{2}
\underbrace{\vkqqqq{\vkwww{F_{1,1}(x_1,y_1)F_{1,1}(x_1,y'_1)F_{1,3}(x_1,y_3)}_{x_1}^2}_{y_1,y'_1,y_3}}_{\mathcal{A}^{(10)}} \\
& \ \ + \frac{1}{2}
\underbrace{\vkqqqq{\vkwww{F_{2,1}(x_2,y_1)F_{2,1}(x_2,y'_1)F_{2,3}(x_2,y_3)}_{x_2}^2}_{y_1,y'_1,y_3}}_{\mathcal{A}^{(11)}}
\end{align*}}
Since $\mathcal{A}^{(9)}$ is analogous to $\mathcal{A}^{(4)}$,
the algorithm loops (with a small ``weight'' $\delta$) and we continue with $\mathcal{A}^{(8)}$ only.
Also, by symmetry, we can consider $\mathcal{A}^{(10)}$ and disregard $\mathcal{A}^{(11)}$.

To control $\mathcal{A}^{(8)}$ we define
{\small\begin{align*}
\mathcal{B}^{(8)} & = \vkqqqq{\vkqqq{F_{1,1}(x_1,y_1)F_{2,1}(x_2,y_1)}_{y_1}^3}_{x_1,x_2} \\
& = \vkqqqq{\vkqqq{F_{1,1}(x_1,y_1)F_{1,1}(x_1,y'_1)F_{1,1}(x_1,y''_1)}_{x_1}
\vkqqq{F_{2,1}(x_2,y_1)F_{2,1}(x_2,y'_1)F_{2,1}(x_2,y''_1)}_{x_2}}_{y_1,y'_1,y''_1}.
\end{align*}}
Every term in $\Box\mathcal{B}^{(8)}$ different from $\mathcal{A}^{(8)}$ is of the shape
{\small\begin{align*}
\mathcal{A}^{(12)} & \!= \Big(\vkwww{F_{1,1}(x_1,y_1)F_{1,1}(x_1,y'_1)F_{1,1}(x_1,y''_1)}_{x_1}
\vkwww{F_{2,1}(x_2,y_1)F_{2,1}(x_2,y'_1)F_{2,1}(x_2,y''_1)}_{x_2}\Big)_{y_1,y'_1,y''_1}
\end{align*}}
and is estimated as
{\small\begin{align*}
|\mathcal{A}^{(12)}| & \leq \ \frac{1}{2}
\underbrace{\vkqqqq{\vkwww{F_{1,1}(x_1,y_1)F_{1,1}(x_1,y'_1)F_{1,1}(x_1,y''_1)}_{x_1}^2}_{y_1,y'_1,y''_1}}_{\mathcal{A}^{(13)}} \\
& \ \ +\frac{1}{2}\vkqqqq{\vkwww{F_{2,1}(x_2,y_1)F_{2,1}(x_2,y'_1)F_{2,1}(x_2,y''_1)}_{x_2}^2}_{y_1,y'_1,y''_1}.
\end{align*}}
Finally, $\mathcal{A}^{(13)}$ corresponds to a partition with the smallest rank.
We introduce
{\small\begin{align*}
\mathcal{B}^{(13)} & = \vkqqqq{\vkqqq{F_{1,1}(x_1,y_1)F_{1,1}(x_1,y'_1)F_{1,1}(x_1,y''_1)}_{x_1}^2}_{y_1,y'_1,y''_1}
= \vkqqqq{\vkqqq{F_{1,1}(x_1,y_1)F_{1,1}(x'_1,y_1)}_{y_1}^3}_{x_1,x'_1}
\end{align*}}
and observe that by Theorem \ref{vktheoremterms} each term in $\Box\mathcal{B}^{(13)}$ different from
$\mathcal{A}^{(13)}$ is of the form
$$ \Big(\vkwww{F_{1,1}(x_1,y_1)F_{1,1}(x'_1,y_1)}_{y_1}^2
\vkqqq{F_{1,1}(x_1,y_1)F_{1,1}(x'_1,y_1)}_{y_1}\Big)_{x_1,x'_1}. $$
Sum of these terms is nonnegative by Lemma \ref{vklemmaterms}
and can be discarded as it only increases $\Box\mathcal{B}^{(13)}$.

On the other hand, to deal with $\mathcal{A}^{(10)}$ we define
{\small\begin{align*}
\mathcal{B}^{(10)} & =
\vkqqqq{\vkqqq{F_{1,1}(x_1,y_1)F_{1,1}(x_1,y'_1)F_{1,3}(x_1,y_3)}_{x_1}^2}_{y_1,y'_1,y_3} \\
& = \vkqqqq{\vkqqq{F_{1,1}(x_1,y_1)F_{1,1}(x'_1,y_1)}_{y_1}^2
\vkqqq{F_{1,3}(x_1,y_3)F_{1,3}(x'_1,y_3)}_{y_3}}_{x_1,x'_1}.
\end{align*}}
Every term in $\Box\mathcal{B}^{(10)}$ other than $\mathcal{A}^{(10)}$
takes either the shape
{\small $$ \Big(\vkwww{F_{1,1}(x_1,y_1)F_{1,1}(x'_1,y_1)}_{y_1}^2
\vkqqq{F_{1,3}(x_1,y_3)F_{1,3}(x'_1,y_3)}_{y_3}\Big)_{x_1,x'_1}, $$}
or the shape
{\small\begin{align*}
& \Big(\vkwww{F_{1,1}(x_1,y_1)F_{1,1}(x'_1,y_1)}_{y_1}
\vkqqq{F_{1,1}(x_1,y_1)F_{1,1}(x'_1,y_1)}_{y_1}
\vkwww{F_{1,3}(x_1,y_3)F_{1,3}(x'_1,y_3)}_{y_3}\Big)_{x_1,x'_1}.
\end{align*}}
Sum of the former terms is nonnegative by Lemma \ref{vklemmaterms},
while the latter ones are treated similarly as $\mathcal{A}^{(6)}$.

A concrete Bellman function $\mathcal{B}$ satisfying $|\mathcal{A}|\leq\Box\mathcal{B}$
and (\ref{vkeqbcontrol}) under Normalization (\ref{vkeqfnormalize}) can be
assembled as a linear combination of all averaging paraproduct-type terms that appear in the proof,
including the omitted ones.
Estimate (\ref{vkeqmainteleestimate}) gives (\ref{vkeqsingletreenormalized}) once again.

\section{Several remarks}
\label{vksec6remarks}

\subsection{Comments on the exponent range}

Already the twisted paraproduct shows that the range of exponents in (\ref{vkeqlpestimate}) is not optimal.
Our restriction of the range comes from the fact that the proof of (\ref{vkeqsingletree})
for a single connected component gradually reduces the term $\mathcal{A}$
to paraproduct-type terms with ``more and more'' repeating functions.
In the worst case we end up with only one function, say $F_{1,1}$, appearing $mn$ times.
This means that the byproduct of such proof is
$$ \sum_{Q\in\mathcal{T}} |Q| \,\big|\mathcal{A}_{Q}(\underbrace{F_{1,1},F_{1,1},\ldots,F_{1,1}}_{mn})\big|
\ \lesssim_{m,n} |Q_\mathcal{T}| \,\Big(\max_{Q\in\mathcal{T}\cup\mathcal{L}(\mathcal{T})}
\vkqqq{F_{1,1}^{d}}_{Q}^{1/d}\Big)^{mn} , $$
where
$$ \mathcal{A}_{I\times J}(F_{1,1},\ldots,F_{1,1}) =
\vkqqqq{\vkwwww{\prod_{\substack{1\leq i\leq m\\ 1\leq j\leq n}}\!\! F_{1,1}(x_i,y_j)
}_{\substack{x_i\in I \textrm{ for } i\in S\\ y_j\in J \textrm{ for } j\in T}\,}
}_{\substack{x_i\in I \textrm{ for } i\in S^c\\ y_j\in J \textrm{ for } j\in T^c}} $$
and $d=d_{1,1}$.
Taking $\mathcal{T}$ to be a single square $\big\{[0,2)^2\big\}$
and $F_{1,1}$ to be a nonnegative elementary tensor $F_{1,1}(x,y)=f(x)g(y)$ supported on $[0,1)^2$
we conclude
$$ \vkqqq{f^n}_{[0,1)}^m \vkqqq{g^m}_{[0,1)}^n \,\lesssim_{m,n}
\Big(\vkqqq{f^{d}}_{[0,1)}^{1/d} \vkqqq{g^{d}}_{[0,1)}^{1/d}\Big)^{mn} , $$
which is only possible if $d\geq\max\{m,n\}$.

However, the general method based on Section \ref{vksec2bellmanfunctions} does not share those limitations
and one can possibly find a Bellman function $\mathcal{B}$ that proves (\ref{vkeqsingletree})
with smaller exponents $d_{i,j}$.
In some particular cases the procedure of reducing the general term
to the ``simpler'' ones can be employed recursively rather than inductively,
thus avoiding to deal with unnecessary paraproduct-type terms.

As an example consider a cycle of length $4n$, i.e.\@
$$ E=\big\{(1,1),(1,2),(2,2),(2,3),(3,3),(3,4),\ldots,(2n,2n),(2n,1)\big\} $$
and suppose $1,2\in S$.
We can estimate the general term by separating brackets $\vkq{\cdot}_{x_i}$ or $\vkw{\cdot}_{x_i}$ for odd and even $i$:
\begin{align*}
|\mathcal{A}|\, & \leq \ \,\frac{1}{2}
\vkqq{\vkw{F_{1,1}(x_1,y_1)F_{1,2}(x_1,y_2)}_{x_1}^2}_{y_1,y_2}\!
\vkqq{\vkq{F_{3,3}(x_3,y_3)F_{3,4}(x_3,y_4)}_{x_3}^2}_{y_3,y_4}\ldots \\
& \ \ +\frac{1}{2}
\vkqq{\vkw{F_{2,2}(x_2,y_2)F_{2,3}(x_2,y_3)}_{x_2}^2}_{y_2,y_3}\!
\vkqq{\vkq{F_{4,4}(x_4,y_4)F_{4,5}(x_4,y_5)}_{x_4}^2}_{y_4,y_5}\ldots \,.
\end{align*}
This gives (\ref{vkeqsingletree}) with $d_{i,j}=2$
and consequently also (\ref{vkeqlpestimate}) in the range \,$2\!<\!p_{i,j}\!<\!\infty$,\,
rather than \,$2n\!<\!p_{i,j}\!<\!\infty$.

Determining the full range of exponents $(p_{i,j})_{(i,j)\in E}$
for which (\ref{vkeqlpestimate}) holds remains an open problem.
The only counterexamples we know are at the boundary of the Banach simplex,
see \cite{K1}.
Certain extensions are possible in the class of mixed-norm $\mathrm{L}^p$ spaces,
but they still face similar restrictions.

\subsection{A remark on non-bipartite graphs}

We close this section by commenting on our setup.
Even though it is possible to associate multilinear forms to general undirected graphs,
the method cannot be applied to forms arising from graphs that are not bipartite.
For instance, a multilinear form associated to a \emph{triangle},
i.e.\@ a cycle of length $3$, could be
$$ \Lambda_{\triangle}(F,G,H)
=\!\sum_{\substack{I,J,K\in\mathcal{D}\\ |I|=|J|=|K|\\ c(I,J,K)=0}}\!
|I|^{1/2}\!\int_{\mathbb{R}^3}\! F(x,y) G(y,z) H(z,x)
\,\psi^{\mathrm{D}}_{I}(x) \psi^{\mathrm{D}}_{J}(y) \psi^{\mathrm{D}}_{K}(z) \,dx dy dz , $$
where \,$c(I,J,K)=0$\, is some constraint, making the sum effectively indexed by only two of the intervals.
Such forms seem to share many characteristics with the
\emph{two-dimensional ``triangular'' Hilbert transform},
$$ \Lambda_{\triangle\mathrm{BHT}}(F,G,H)
=\int_{\mathbb{R}^2}\Big(\mathrm{p.v.}\int_{\mathbb{R}}F(x-t,y)\,G(x,y-t)\,\frac{dt}{t}\Big) \,H(x,y) \,dx dy , $$
arising from ``singular bilinear averages'' introduced in \cite{DT}
and for which no $\mathrm{L}^p$ bounds are known at the time of writing.

\section*{Acknowledgments}

I am indebted to Professor Christoph Thiele, my advisor at UCLA,
for his invaluable support while this research was performed.
I would like to thank Professor John Garnett
for correcting a historical reference on paraproducts.
I am also grateful to Professor Alexander Volberg
for a useful discussion at a very early stage of this project.

\end{document}